\documentclass[12pt]{article} 
\usepackage{amsmath, amssymb}
\usepackage{amsthm} 
\usepackage{mathtools}
\usepackage{mathrsfs}
\usepackage{enumerate}
\usepackage{empheq}
\usepackage{esint}
\usepackage{bbm}
\usepackage{color}
\usepackage[T1]{fontenc}
\usepackage{textcomp}
\usepackage[abbrev]{amsrefs}
\usepackage[dvipdfmx]{hyperref}
\usepackage[UKenglish]{babel}
\usepackage[UKenglish]{isodate}
\numberwithin{equation}{section}
\theoremstyle{plain} 
\newtheorem{theorem}{Theorem}[section]
\newtheorem{lemma}[theorem]{Lemma}
\newtheorem{corollary}[theorem]{Corollary}
\newtheorem{proposition}[theorem]{Proposition}

\theoremstyle{definition} 
\newtheorem{definition}[theorem]{Definition}
\newtheorem{remark}[theorem]{Remark}

\def\N{\mathbb N}

\def\R{\mathbb R}

\def\F{\mathcal F}

\newcommand{\iO}{\int_\Omega}
\newcommand{\iB}{\int_{\Omega}}

\usepackage{bbm}
\usepackage{color}

\allowdisplaybreaks

\makeatletter
\def\address#1#2{\begingroup
\noindent\parbox[t]{7.8cm}{%
\small{\scshape\ignorespaces#1}\par\vskip1ex
\noindent\small{\itshape E-mail address}%
\/: #2\par\vskip4ex}\hfill%
\endgroup}%
\makeatother
\title{\uppercase{A sufficient condition for absence of mass quantization in a chemotaxis system with local sensing}} 
\author{
\bigskip \\
\textsc{Yuri Soga} 
}
\date{\today} 
\begin{document}

\maketitle


\begin{abstract}
We analyze blowup solutions in infinite time of the Neumann boundary value problem for the fully parabolic chemotaxis system with local sensing:
\begin{equation*}
\begin{cases}
 u_t = \Delta(e^{-v}u)\qquad &\mathrm{in}\ \Omega \times (0,\infty),\\
 v_t = \Delta v -v + u\qquad &\mathrm{in}\ \Omega \times (0,\infty),
\end{cases}
\end{equation*}
where $\Omega$ is a ball in two-dimensional space and with nonnegative radially symmetric initial data. In the case of the Keller--Segel system which has a similar mathematical structure with our system, it was shown that 
the solutions blow up in finite time if and only if $L\log L$ for the first component $u$ diverges in finite time.
 On the other hand, focusing on the variational structure induced by a signal-dependent motility function $e^{-v}$, we show that an unboundedness of $\int_\Omega e^v dx$ for the second component $v$ gives rise to blowup solutions in infinite time under the assumption of radial symmetry. Moreover we prove mass concentration phenomena at the origin. It is shown that the radially symmetric solutions of our system develop a singularity like a Dirac delta function in infinite time. 
Here we investigate the weight of this singularity.
Consequently it is shown that mass quantization may not occur; that is, the weight of the singularity can exceed $8\pi$ under the assumption of a uniform-in-time lower bound for the Lyapunov functional. This type of behavior cannot be observed in the Keller--Segel system.
\end{abstract}
\newpage

\section{Introduction}

In this paper, we consider global-in-time solutions of the $0$-Neumann boundary value problem for the chemotaxis system:
\begin{equation}\label{p}\tag{P}
\begin{cases}
u_t = \Delta(e^{-v}u)\qquad &\mathrm{in}\ \Omega \times (0,\infty),\\
v_t = \Delta v -v + u\qquad &\mathrm{in}\ \Omega \times (0,\infty),\\
\partial_\nu u = \partial_\nu v = 0 \qquad &\mathrm{on}\ \partial \Omega \times (0,\infty),\\
(u(\cdot,0),v(\cdot, 0)) = (u_0(\cdot),v_0(\cdot))\qquad &\mathrm{in}\ \Omega,
\end{cases}
\end{equation}
where $\Omega$ is a ball in two-dimensional space, $\partial_\nu$ denotes the derivative with respect to the outward normal of $\partial\Omega$, and the initial data $(u_0, v_0)$ satisfies
\begin{equation}\label{c}
\begin{cases}
u_0 \in W^{1,\infty}(\Omega),\quad u_0 \ge 0\ (u_0 \not\equiv 0)\quad &\mathrm{in}\ \Omega,\\
v_0 \in W^{1,\infty}(\Omega),\quad v_0 > 0 &\mathrm{in}\ \Omega.
\end{cases}
\end{equation}

This model is motivated from biological background (see \cite{FTL2012, LFL2011}). From the viewpoint of mathematics, the system \eqref{p} shares similar mathematical structures as the Keller–
Segel system (see \cite{KS1970, BBTW2015}):

\begin{equation}\label{KS}\tag{KS}
\begin{cases}
 u_t = \Delta u - \nabla \cdot (u\nabla v)\qquad &\mathrm{in}\ \Omega \times (0,\infty),\\
 v_t = \Delta v -v + u\qquad &\mathrm{in}\ \Omega \times (0,\infty).
\end{cases}
\end{equation}
Indeed, both systems admit the same stationary problem and have a Lyapunov functional given by
\begin{align*}
\F (t) := \iO (u\log u - uv) dx + \dfrac{1}{2}\|\nabla v\|_{L^2(\Omega)}^2 + \dfrac{1}{2}\|v\|_{L^2(\Omega)}^2,
\end{align*}
which was established by \cite{B1998, NSY1997} for the Keller--Segel system and by \cite{HZ2020, FJ2021} for our system.
In the case of our system, the Lyapunov functional satisfies the energy identity
\begin{align}
\dfrac{d}{dt}\F (t) + \iO ue^{-v}|\nabla (\log u - v)|^2 dx + \|v_t\|_{L^2(\Omega)}^2 = 0.\label{lya}
\end{align}
This Lyapunov functional is deeply related to a long-time behavior of solutions. The combination of the
Lyapunov functional and the Trudinger–Moser inequality (see \cite{CY1988, MJ1970}) implies the mass critical phenomena, which seems to be one of the remarkable points of the study for large-time behaviors of solutions to the Keller--Segel system. Indeed, in the two-dimensional and radially symmetric setting,
it is known that solutions of the Keller--Segel system exhibit different long-time behavior depending on whether the size of mass is below or above the threshold value $8\pi$, which is called ``$8\pi$-problem.'' To be precise, the following results are known.
\begin{itemize}
\item if $\|u_0\|_{L^1(\Omega)} < 8\pi$, the corresponding classical solution of \eqref{p} exists globally and remains bounded (see \cite{NSY1997, B1998, M2013}).
\item There exists some initial data with $\|u_0\|_{L^1(\Omega)} > 8\pi$ such that the corresponding solution blows up in finite time (see \cite{M2020_1, M2020_2, HV1996, HV1997, N1995}).
\end{itemize}
On the other hand, although our system \eqref{p} and the Keller--Segel system have same mathematical features such as the Lyapunov functional and the stationary problem, the mass critical phenomena different from that of the Keller--Segel system were exhibited by \cite{HZ2020, FJ2020, FJ2021, FJ2022}:
\begin{itemize}
\item The classical solution of \eqref{p} always exists globally in time regardless of the size
of the initial mass $\|u_0\|_{L^1(\Omega)}$ and moreover if $\|u_0\|_{L^1(\Omega)} < 4\pi$ (or $8\pi$ in the radially symmetric case), the corresponding classical solution of \eqref{p} is uniformly bounded in time.
\item There exists some initial data with $\|u_0\|_{L^1(\Omega)} \in (4\pi, \infty) \setminus 4\pi \N$ (or $\|u_0\|_{L^1(\Omega)} > 8\pi$ in the radially symmetric case) such that the corresponding solution of \eqref{p} blows up
in infinite time.
\end{itemize}
In addition, even in the case of chemotaxis systems containing various motility functions, many studies are carried out (see \cite{AY2019, JL2021, FJ2021_2, TW2017, FS2022_1, FS2022_2}).
In this paper, we focus on infinite time blowup solutions of \eqref{p}, which differ from the long-time behavior of solutions to the Keller--Segel system. 

In the detailed analysis of blowup solutions,  the solutions of \eqref{KS} that blow up in finite time develop a singularity like a Dirac delta function, a phenomenon known as ``chemotactic collapse.'' This behavior was initially demonstrated in the pioneering works of Herrero--Vel\'{a}zquez \cite{HV1996, HV1997}, followed by Nagai--Senba--Suzuki \cite{NSS2000} for the fully parabolic Keller--Segel system, and Senba--Suzuki \cite{SS2001} for the parabolic-elliptic Keller–Segel system. 
Especially in \cite{NSS2000, SS2001}, it was proved that the finite time blowup solution $u$ satisfies
\begin{align*}
u(\cdot, t) \stackrel{*}{\rightharpoonup} m \delta(0) + f\ \mathrm{in}\ \mathscr{M}(\overline{\Omega})\  \mathrm{as}\ t \to T_\mathrm{max} < \infty,
\end{align*}
where the weight $m$ is the constant such that $m \ge 8\pi$ and $\mathscr{M}(\overline{\Omega})$ is denoted by the dual space of $C(\overline{\Omega})$. It is an open problem whether mass quantization ($m = 8\pi$) occurs or not ($m > 8\pi$) in the Keller--Segel system. On the other hand, in the parabolic-elliptic Keller--Segel system, it was shown by Senba--Suzuki \cite{SS2001_2, SS2002} that mass quantization of collapse occurs if the solution blows up in infinite time. The key method in \cite{SS2001_2, SS2002} is a blowup criterion derived from the second moment law. Subsequently, Suzuki \cite{ST2005} showed mass quantization of collapse for finite time blowup solutions by using the method of backward self-similar transformation. However, these methods cannot be directly applied to the fully parabolic Keller--Segel system. Therefore, new techniques are required to analyze mass concentration phenomena in the Keller--Segel system, yet the problem remains open.

In the author's previous work \cite{YS2025}, an attempt was made to overcome the problem of mass quantization for infinite time blowup solutions to the chemotaxis system with indirect signal production:
\begin{equation}\notag
\begin{cases}
 u_t = \Delta u - \nabla \cdot (u \nabla w)\qquad &\mathrm{in}\ \Omega \times (0,\infty),\\
 v_t = -v + u\qquad &\mathrm{in}\ \Omega \times (0,\infty),\\
 w_t = \Delta w -w + v \qquad &\mathrm{in}\ \Omega \times (0,\infty).
\end{cases}
\end{equation}
This chemotaxis system exhibits the same mass critical phenomena as \eqref{p} (see \cite{L2019}). In \cite{YS2025}, motivated by \cite{NSS2000, SS2001}, it is shown that $L\log L$-boundedness for the first component $u$ determines whether solutions remain bounded in time or blow up in infinite time. The fundamental idea of the proof is global-in-time energy estimates, in contrast to local-in-time  estimates in \cite{NSS2000, SS2001}.
Then, the similar mass concentration phenomena as \cite{NSS2000, SS2001} is proved. Moreover it is shown that mass quantization may not occur under the assumption that the Lyapunov functional has a lower bound along the solution curve. More precisely, under this assumption, it holds that there exists a sequence $\{t_k\}$ of time with $t_k \to \infty$ as $k \to \infty$ such that
\begin{align*}
u(\cdot, t_k) \stackrel{*}{\rightharpoonup} \|u_0\|_{L^1(\Omega)} \delta(0)\ \mathrm{in}\ \mathscr{M}(\overline{\Omega})\ \mathrm{as}\ k \to \infty.
\end{align*}
Hence the weight of a delta function can exceed $8\pi$.
In contrast, Mizoguchi \cite{M2020_1} showed that in the Keller--Segel system, the Lyapunov functional is uniformly bounded from below in time and a concentration of the mass in infinite time cannot happen with $\|u_0\|_{L^1(\Omega)} > 8\pi$. Therefore this type of behavior can never be observed in the Keller--Segel system.

Inspired by the previous work \cite{YS2025}, we initiated this study by investigating whether a similar argument could be applied to our system \eqref{p} as well. However, it was impossible to apply the argument in \cite{YS2025} directly due to the variational structure of our system.  Indeed, in the case of the Keller--Segel system (see \cite{NSS2000, SS2001, SS2001_2}) (and the chemotaxis system with indirect signal production (see \cite{YS2025})), the solutions blow up in finite time (resp. in infinite time) if and only if $L\log L$ for the density of cells $u$ blows up in finite time (resp. in infinite time). The key ingredient for the proof above is the variant of the Sobolev inequality which were established in \cite{BHN1994, NSY1997}. On the other hand, the signal-dependent motility function $e^{-v}$ contained in our system \eqref{p} prevents this inequality from being applicable (see Section \ref{sec3}). 
This reason and the method of the proof in \cite{FJ2021} led us to predict that $L\log L$ for the component $u$ is not essential factor leading to grow-up solution of \eqref{p}.  
Consequently, focusing on the second component $v$ in light of the variational structure of \eqref{p},
we succeed in characterizing grow-up solutions distinct from \cite{YS2025} and proving mass concentration phenomena.
Let us introduce the following definition.

\begin{definition}
Let $(u,v)$ be a classical solution to \eqref{p} in $\Omega \times (0, \infty)$. 
\begin{enumerate}
\item[\rm{(i)}] We shall call $(u,v)$ a grow-up solution if it holds that
\begin{align*}
\limsup_{t \to \infty}\|u(t)\|_{L^\infty(\Omega)} = \infty.
\end{align*}
\item[\rm{(ii)}] We shall call $x_0 \in \overline{\Omega}$ a grow-up point if there exist $\{t_k\} \subset (0,\infty)$ and $\{x_k\} \subset \overline{\Omega}$ such that
\begin{align*}
t_k \to \infty,\quad x_k \to x_0,\quad u(x_k, t_k) \to \infty\ \mathrm{as}\ k \to \infty.
\end{align*}
Moreover, we shall denote by $\mathcal{G}$ the set of grow-up points.
\end{enumerate}
\end{definition}
We are interested in studying the behavior of grow-up solutions to \eqref{p} near grow-up points. Therefore, the existence of grow-up points is essential when considering grow-up solutions (see Proposition \ref{prop:groe}).

\noindent
\textbf{Main results.}
Our main results are stated as follows. In this paper, we denote $B_r(x_0)$ by $\{x \in \R^2; |x-x_0| < r\}$ and write $B_r$ instead of $B_r(0)$.
\begin{theorem}\label{th:1}
Let $\Omega = B_R \subset\R^2$ with some $R > 0$ and $a > 1$. Assume that $(u_0,v_0)$ is a couple of radially symmetric functions satisfying \eqref{c}. Then the corresponding classical solution $(u,v)$ is a grow-up solution if and only if the second component $v$ satisfies 
\begin{align*}
\limsup_{t \to \infty}\iO e^{av(x,t)} dx = \infty.
\end{align*}
\end{theorem}

\begin{remark}\label{remark:1}
As noted, we cannot apply the argument in \cite{NSS2000, SS2001} for the Keller--Segel system and \cite{YS2025} for the chemotaxis system with indirect signal production due to the influence of the variational structure of our system. Theorem \ref{th:1} is shown through the development of a new inequality to replace the inequality in \cite{BHN1994} and the utilization of an auxiliary function which presents the pointwise upper bound estimate for $v$. This auxiliary function was introduced by \cite{FJ2020, FJ2021}.
This theorem sheds new light on the dynamics of solutions in chemotaxis systems.
\end{remark}

Motivated by the author's previous work \cite{YS2025} and thanks to Theorem \ref{th:1}, we investigate mass concentration phenomena.

\begin{corollary}\label{th:2}
Let $\Omega = B_R \subset \R^2$ with some $R > 0$. Assume that $(u_0,v_0)$ is a couple of radially symmetric functions satisfying \eqref{c}. If $(u,v)$ is a grow-up solution of \eqref{p}, then there exist a sequence $\{t_k\}$ of time with $t_k \to \infty$ as $k \to \infty$ and $m \ge 8\pi$, together with a nonnegative function $f \in L^1(\Omega) \cap C(\overline{\Omega}\setminus \{0\})$ such that
\begin{align*}
u(\cdot, t_k) \stackrel{*}{\rightharpoonup} m \delta(0) + f\ \mathrm{in}\ \mathscr{M}(\overline{\Omega})\  \mathrm{as}\ k \to \infty.
\end{align*}
In other words, it holds that for any $\xi \in C(\overline{\Omega})$,
\begin{align*}
\iB u(x,t_k)\xi(x) dx \rightarrow m\xi(0) + \iB f(x)\xi(x) dx\ \mathrm{as}\ k \to \infty.
\end{align*}
\end{corollary}

\begin{remark}
In our system, the assumption of radial symmetry for initial data ensures that the grow-up point is the only origin. However, it is not trivial. For example, we may consider the blowup solutions in finite time of the Dirichlet boundary value problem for the semilinear heat equation:
\begin{equation}\label{heat}
\begin{cases}
 u_t = \Delta u + |u|^{p-1}u &\mathrm{in}\ \Omega \times [0,T),\\
 u= 0 &\mathrm{on}\ \partial\Omega \times [0,T),
\end{cases}
\end{equation}
where $\Omega = B_R \subset \R^2$ for sufficiently large $R > 0$ and $p > 1$. If we consider initial data vanishing near $0$ but large near $|x| = R$, then the solution $u$ of \eqref{heat} does not blow up at the origin (see \cite{GK1989}). The reader is referred to \cite{MF1990} for 
Schr\"{o}dinger equations. For this reason, we need to show $\mathcal{G} = \{0\}$ under the assumption of radial symmetry (see Proposition \ref{pro:rad0}).
\end{remark}

It is an open problem whether the mass quantization ($m = 8\pi$) occurs or not ($m > 8\pi$) in Corollary \ref{th:2}. Here focusing on the boundedness of the Lyapunov functional, we will present the following theorem for this challenge, which cannot be observed in the Keller--Segel system.

\begin{corollary}\label{th:3}
Let $\Omega = B_R \subset \R^2$ with some $R > 0$. Assume that $(u_0,v_0)$ be a couple of radially symmetric functions satisfying \eqref{c} and $\|u_0\|_{L^1(\Omega)} > 8\pi$. If a grow-up solution $(u,v)$ of \eqref{p} satisfies 
\begin{align}
\inf_{t \ge 0} \F (t) > -\infty,\label{eq:i3}
\end{align}
then there exists a sequence $\{t_k\}$ of time with $t_k \to \infty$ as $k \to \infty$ such that
\begin{align*}
u(\cdot, t_k) \stackrel{*}{\rightharpoonup} \|u_0\|_{L^1(\Omega)} \delta(0)\ \mathrm{in}\ \mathscr{M}(\overline{\Omega})\ \mathrm{as}\ k \to \infty.
\end{align*}
\end{corollary}

\begin{remark}
Corollary \ref{th:2} and Corollary \ref{th:3} establish the same type of results as in \cite{YS2025}, but for a different equation. Hence these results represent a new development for our chemotaxis system \eqref{p}. Moreover, the method of proof employed in this paper is completely different from \cite{YS2025}.
As stated in Remark \ref{remark:1}, we clarify the characterization of grow-up solutions to \eqref{p} in  Theorem \ref{th:1} based on a novel idea. Since this theorem leads to Corollary \ref{th:2} and subsequently to Corollary \ref{th:3}, we highlight that our results are not a direct consequence of the argument in the author's previous work \cite{YS2025}.
\end{remark}

\begin{remark}
Whether the assumption \eqref{eq:i3} holds for a grow-up solution of \eqref{p} is an open problem, yet we emphasize that Corollary \ref{th:3} implies new difference about large-time behaviors of solutions to \eqref{p} and the Keller--Segel system. 
\begin{itemize}
\item If there exists a grow-up solution of \eqref{p} satisfying the assumption \eqref{eq:i3} in Corollary \ref{th:3} or the grow-up solution of \eqref{p} always satisfies the assumption, then the mass quantization of collapse does not occur and the weight of a delta function completely equals to the mass $\|u_0\|_{L^1(\Omega)}$. These results are different from that of the Keller--Segel system. 
\item If the assumption \eqref{eq:i3} in Corollary \ref{th:3} is false, then we can conclude that every radially symmetric grow-up solution of \eqref{p} satisfies $\inf_{t > 0}\F (t) = -\infty$. In contrast, in the Keller--Segel system, it was shown by Mizoguchi \cite{M2020_1} that the radially symmetric global-in-time solution of the Keller--Segel system always satisfies a uniform-in-time lower bound for the Lyapunov functional. Therefore, This result is unlike for the Keller–Segel system.
\end{itemize}
Consequently, although our system and the Keller--Segel system share common features with respect to the stationary problem and the Lyapunov functional, we can conclude that there are differences in the global behavior of solutions.
\end{remark}

\noindent
\textbf{Plan of this paper.} This paper is organized as follows. In Section \ref{sec:pre}, we prepare
some fundamental properties of the solution to \eqref{p} and state some useful lemmas. In Section \ref{sec3}, we state the difference in the variational structures of \eqref{p} and the Keller--Segel system and the associated challenges in an a priori estimate. Then in Subsection \ref{subsec:1}, we show the key functional inequality to proceed a priori estimates in a localized area. After that, we enter the stage of proving theorem \ref{th:1} and Corollaries. In Subsection \ref{subsec:2}, we firstly prove that the grow-up point is the only origin and secondly prove the concentration lemma based on the original idea, and finally obtain Theorem \ref{th:1}. In Subsection \ref{subsec:3}, we show the mass concentration phenomena following the similar approach as \cite{YS2025}.

\section{Preliminaries}\label{sec:pre}

In this section, we collect key ingredients for the proof of our results. We begin by introducing the local existence and uniqueness of classical solutions. More strongly we emphasize that the unique classical solution of \eqref{p} exists globally in time without dependence on the initial mass, which is shown by Fujie--Jiang \cite{FJ2021}.
The proof for local existence and uniqueness of a non-negative classical solution to \eqref{p} can be completed by  an application of the Amann theory \cite{AH1990, AH1993} as noted in \cite{FJ2021}, or an employment of the standard parabolic regularity theory and the Schauder fixed point theorem as outlined in \cite{AY2019}. The positivity follows from the strong maximum principle for classical solutions. As for the proof of the global-in-time solution, we may refer the reader to \cite{FJ2021}. The properties \eqref{mass:1} and \eqref{mass:2} result from the $0$-Neumann boundary condition.

\begin{proposition}\label{loc:1}
Let $\Omega$ be a bounded domain in $\R^2$ with smooth boundary and let $(u_0,v_0)$ be as in \eqref{c}. Then there exists a unique positive classical solution of \eqref{p}
\begin{align*}
(u,v) \in (C(\overline{\Omega} \times [0,\infty)) \cap C^{2,1}(\overline{\Omega} \times (0,\infty)))^2.
\end{align*}
Moreover the following mass conservation law holds 
\begin{align}
\iO u(x,t) dx = \iO u_0(x) dx\ \mathrm{for\ all}\ t \in (0,\infty)\label{mass:1}
\end{align}
and it holds that
\begin{align}
\int_\Omega v(x,t) dx \leq \max\{\|u_0\|_{L^1(\Omega)}, \|v_0\|_{L^1(\Omega)}\}\ \mathrm{for\ all}\ t \in (0,\infty).\label{mass:2}
\end{align}
\end{proposition}

We obtain a strictly positive uniform-in-time lower bound for the second component $v$ of \eqref{p} under the assumption that $v_0$ is a strictly positive function in $\overline{\Omega}$. 

\begin{lemma}[{\cite[Lemma 3.1]{FS2018}}]\label{lem:below}
Assume that the initial data $(u_0,v_0)$ satisfies \eqref{c}. If $(u,v)$ is the classical solution of \eqref{p} in $\Omega \times (0, \infty)$,  there exists $v_* > 0$ such that for all $t \ge 0$,
\begin{align*}
\inf_{x \in \Omega}v(x,t) \ge v_* > 0.
\end{align*}
\end{lemma}

The next proposition is a fairly straightforward of semigroup properties.

\begin{proposition}
Suppose that $(u, v)$ be a solution to \eqref{p} in $\Omega \times (0, \infty)$ with \eqref{c}.
Then it holds that:
\begin{enumerate}
\item[\rm{(a)}] Let $p \in (1,\infty)$, then there exists $C_1 > 0$ such that for all $t \in (0,\infty)$,
\begin{align}\label{semi:1}
\|v(t)\|_{L^p(\Omega)} \leq C_1(\|u_0\|_{L^1(\Omega)} + \|v_0\|_{L^p(\Omega)}).
\end{align}
\item[\rm{(b)}] Let $p \in (1,2)$, then there exists $C_2 > 0$ such that for all $t \in (0,\infty)$,
\begin{align}\label{semi:2}
\|\nabla v(t)\|_{L^p(\Omega)} \leq C_2(\|u_0\|_{L^1(\Omega)} + \|\nabla v_0\|_{L^p(\Omega)}).
\end{align}
\end{enumerate}
\end{proposition}

In addition to the above information, we get a pointwise estimate for $v$ in the case $(u,v)$ is a radially symmetric solution of \eqref{p}.

\begin{proposition}[{\cite[Lemma 3.2]{WM2013}}]\label{pro:322}
Let $(u, v)$ be a radially symmetric solution to \eqref{p} in $\Omega \times (0, \infty)$ with \eqref{c}. Then, for $p \in (1,2)$, there exists a positive constant $C$ depending on $p$ such that for all $(x,t) \in \Omega \times (0,\infty)$,
\begin{align*}
v(x,t) \leq C(p) (\|u_0\|_{L^1(\Omega)} + \|v_0\|_{W^{1,p}(\Omega)})|x|^{-\frac{2-p}{p}}.
\end{align*}
By Setting
\begin{align*}
|||v(t)|||_{\kappa} := \sup_{x \in \Omega}||x|^\kappa v(x,t)|\ \mathrm{for}\ t \in (0,\infty),
\end{align*}
where $\kappa = \frac{2-p}{p} \in (0,1)$, it holds that there exists $C > 0$ independent of $t$ such that
\begin{align*}
\sup_{t > 0} |||v(t)|||_{\kappa} \leq C.
\end{align*}

\end{proposition}

We recall the following lemma given in \cite{BS1973} about estimates for the solution of Helmholtz equations.

\begin{lemma}[{\cite[Theorem 22, Lemma 23]{BS1973}}]\label{lem:h1}
Let $\Omega$ be a smooth bounded domain in $\R^n$, $n \ge 1$, and $f \in L^1(\Omega)$ be a nonnegtive function. Then, there exists a unique weak solution of 
\begin{equation}\label{p1}
\begin{cases}
-\Delta z + z = f\qquad &\mathrm{in}\ \Omega,\\
\partial_\nu z = 0 \qquad &\mathrm{on}\ \partial \Omega
\end{cases}
\end{equation}
satisfying $z \in W^{1,p}(\Omega)$ for $1 \leq p < \frac{n}{n-1}$. Moreover, for all $1 \leq p < \frac{n}{n-1}$,
there exists a positive constant $C = C(n,\Omega,p)$ such that 
\begin{align}
\|z\|_{W^{1,p}(\Omega)} \leq C\|f\|_{L^1(\Omega)}.\label{eq:h1}
\end{align}
\end{lemma}

The following lemma is stated in \cite[Lemma 2.2]{AY2019} and \cite[Lemma 1]{FJ2021}.
The estimate naturally follows from the combination of the Sobolev embedding and Lemma \ref{lem:h1}.

\begin{lemma}\label{lem:h2}
Let $\Omega$ be a bounded domain in $\R^n$ with smooth boundary, $n \ge 1$, and $f \in C(\overline{\Omega})$ be a nonnegtive function satisfying $\int_\Omega f dx > 0$. If $z \in C^2(\overline{\Omega})$ is a solution of \eqref{p1}, then there exists a positive constant $C = C(n,\Omega)$ such that $z$ fulfils the pointwise estimate
\begin{align}
z \ge C \iO f dx > 0\quad \mathrm{in}\ \Omega\label{eq:h2}
\end{align}
and for any $1 \leq q < \frac{n}{(n-2)_+}$, there is a positive constant $C = C(n,\Omega,q)$ such that
\begin{align}
\|z\|_{L^q(\Omega)} \leq C\|f\|_{L^1(\Omega)}.\label{eq:h3}
\end{align}
\end{lemma}

Fujie--Jiang \cite{FJ2021} introduced the auxiliary function $w(x,t)$, which is the solution of the following  Helmholtz equation:
\begin{equation*}
\begin{cases}
-\Delta w + w = u\qquad &\mathrm{in}\ \Omega \times (0,\infty),\\
\partial_\nu w = 0 \qquad &\mathrm{on}\ \partial \Omega \times (0,\infty).
\end{cases}
\end{equation*}
Here we notice $w = (I - \Delta)^{-1} u \in C^{2,1}(\overline{\Omega}\times (0,\infty))$ thanks to the standard elliptic theory for the external force term $u$ given by Proposition \ref{loc:1}. We also define $w_0 := (I-\Delta)^{-1} u_0$. This auxiliary function $w(x,t)$ plays an important role in obtaining time-independent a priori estimates of the solution and achieving our main results. We introduce the next lemma by applying \cite[Lemma 7]{FJ2021} with $\gamma (v) = e^{-v}$. This lemma allows us to compare the second component $v$ of \eqref{p} with the auxiliary function $w$, which is shown by employing the comparison principle of parabolic equations.

\begin{lemma}[{\cite[Lemma 7]{FJ2021}}]\label{lem:vupper}
Let $v_*$ be as in Lemma \ref{lem:below}. Then there exists a positive constant $K$ depending on $v_*$ such that
for all $(x,t) \in \Omega \times (0,\infty)$,
\begin{align}
v(x,t) \leq \dfrac{1}{1-e^{-v_*}}(w(x,t) + K).\label{bdd:v}
\end{align}
\end{lemma}

\section{Concentration around grow-up points}\label{sec3}

In the case of two-dimensional Keller--Segel system and some chemotaxis systems, it is well known that the uniform-in-time boundedness of $\iB u\log u dx$ is equivalent to uniform-in-time boundedness of solutions (see \cite{NSY1997, B1998, NSS2000, BBTW2015}). A fundamental method in the proof is the use of the variant of the Sobolev inequality which was developed by Biler--Hebisch--Nadzieja \cite{BHN1994}: given $\varepsilon > 0$, there is a positive constant $C(\varepsilon) > 0$ depending only on $\varepsilon$ and $\Omega$ such that
\begin{align}
\|f\|_{L^3(\Omega)}^3 \leq \varepsilon \|f\|_{H^1(\Omega)}^2 \|f\log |f|\|_{L^1(\Omega)} + C(\varepsilon)\|f\|_{L^1(\Omega)}\label{BHNine}
\end{align}
for any $f \in H^1(\Omega)$, where $\Omega$ is a bounded domain in two--dimensional space. This inequality implies that the term $\|f\|_{L^3(\Omega)}$ can be controlled by the term $\|f\|_{H^1(\Omega)}$ under the boundedness of $L\log L$ for the positive function $f$. Hence this inequality is the important ingredient in a priori estimates.
On the other hand, this inequality does not appear in \cite{FJ2021} when proving uniform-in-time boundedness of solutions to our chemotaxis system with local sensing. Indeed, multiplying the first equation of \eqref{p} by $u$ and integrating by parts, we have
\begin{align*}
\dfrac{1}{2}\dfrac{d}{dt}\|u(t)\|_{L^2(\Omega)}^2 + \iO e^{-v}|\nabla u|^2 dx &= \iO e^{-v} u\nabla v \cdot \nabla u dx.
\end{align*}
Moreover we use the second equation to obtain
\begin{align*}
&\dfrac{1}{2}\dfrac{d}{dt}\|u(t)\|_{L^2(\Omega)}^2 + \iO e^{-v}|\nabla u|^2 dx\\
&= \dfrac{1}{2}\iO e^{-v} u^3 dx -\dfrac{1}{2}\iO e^{-v}v u^2 dx -\dfrac{1}{2}\iO e^{-v}v_t u^2 dx + \dfrac{1}{2}\iO e^{-v}|\nabla v|^2 u^2 dx\\
&\leq \dfrac{1}{2}\iO u^3 dx -\dfrac{1}{2}\iO e^{-v}v u^2 dx -\dfrac{1}{2}\iO e^{-v}v_t u^2 dx + \dfrac{1}{2}\iO e^{-v}|\nabla v|^2 u^2 dx.
\end{align*}
The above inequality means that the function $e^{-v}$ kills the advantage of the good term $\int_\Omega e^{-v}|\nabla u|^2 dx$, and thus even with the inequality \eqref{BHNine}, we cannot control the bad term $\|u\|_{L^3(\Omega)}^3$.
We face the same problem in obtaining time-independent a priori estimates in a localized area. Hence we cannot apply the method in \cite{NSS2000, SS2001, YS2025} directly, thereby we need new techniques.

\subsection{Fundamental inequalities and lemmas}\label{subsec:1}

We first show that there exist grow-up points while examining grow-up solutions of \eqref{p}. This is proved in the same way as in \cite[Lemma 3.1]{YS2025}, thereby we omit the details for brevity here.

\begin{proposition}\label{prop:groe}
Let $(u,v)$ be a grow-up solution of \eqref{p} in $\Omega \times (0,\infty)$. Then it holds that $\mathcal{G} \not= \emptyset$.
\end{proposition}

Motivated by Nagai--Senba--Suzuki \cite{NSS2000} and Senba--Suzuki \cite{SS2001}, we introduce a cut-off function $\varphi$ to attain the localized estimates around the grow-up points. The following lemma agrees with the one given in \cite{NSS2000, SS2001} and \cite[Lemma 2.3]{FS2016}.
\begin{lemma}\label{lem:30}
Let $x_0 \in \overline{\Omega}, n \in \N$ and $r > 0$. Then there exists a function $\varphi = \varphi_{(x_0,r,n)} \in C^\infty_c(\R^2)$ satisfying
\begin{align*}
  &\varphi(x)=
  \begin{cases}
    1\quad \mathrm{in}\ B_r(x_0), \\
    0\quad \mathrm{in}\ \R^2 \setminus B_{2r}(x_0),
  \end{cases}\\
  &0 \leq \varphi \leq 1\ \mathrm{in}\ \R^2,\\
  &\partial_\nu \varphi = 0\ \mathrm{on}\ \partial \Omega,\\
  &|\nabla \varphi| \leq A\varphi^{1-\frac{1}{n}},\ |\Delta \varphi| \leq B\varphi^{1-\frac{2}{n}}\ \mathrm{in}\ \R^2,
\end{align*}
where $A, B$ are positive constants determined by $n, r$.
\end{lemma}

As noted, we have difficulty in using the inequality established by \cite{BHN1994} for our problem. As a result, we cannot  rely on the use of the localized version of the inequality, which was proposed by \cite{NSS2000, SS2001}. To solve this challenge, we present a novel functional inequality taking a signal-dependent motility function into account.
The proof is based on the Sobolev inequality in a bounded domain within a two--dimensional space:
\begin{align}
\|f\|_{L^2(\Omega)}^2 \leq K_{\mathrm{Sob}}^2 (\|\nabla f\|_{L^1(\Omega)}^2 + \|f\|_{L^1(\Omega)}^2)\label{ine:sob}
\end{align}
for any $f \in W^{1,1}(\Omega)$, where $K_{\mathrm{Sob}} > 0$ is a constant depending only on $\Omega$. 

\begin{proposition}\label{pro:ine}
Let $x_0 \in \overline{\Omega}, 0 < r \ll 1$, and $n$ be a sufficiently large, and more $\varphi = \varphi_{(x_0,r, n)}$ be as in Lemma \ref{lem:30}.
Then positive functions $z \in H^1(\Omega)$ and $v \in H^1(\Omega)$ satisfy the following inequality:
\begin{align*}
\iB z^2 \varphi dx \leq 2K_{\mathrm{Sob}}^2 \int_{B_{2r}(x_0) \cap \Omega}z e^v dx \iB e^{-v}\dfrac{|\nabla z|^2}{z}\varphi dx + K_{\mathrm{Sob}}^2 \left(\dfrac{A^2}{2} + 1\right)\|z\|_{L^1(B_{2r}(x_0) \cap \Omega)}^2,
\end{align*} 
where $A$ is given by Lemma \ref{lem:30}.
\end{proposition}

\begin{proof}
Setting $w = z\varphi^\frac{1}{2}$, we calculate
\begin{align*}
\left(\iO |\nabla w| dx\right)^2 &= \left(\iO |(\nabla z)\varphi^\frac{1}{2} + z\nabla \varphi^\frac{1}{2}| dx\right)^2\\
&\leq 2\left(\iO |\nabla z|\varphi^\frac{1}{2}dx\right)^2 + 2\left(\iO z|\nabla \varphi^\frac{1}{2}|dx\right)^2.
\end{align*}
With respect to the first term on the right-hand side of the above inequality, we use the H\"older inequality to obtain
\begin{align*}
2\left(\iO |\nabla z|\varphi^\frac{1}{2}dx\right)^2 &= 2\left(\iO e^{\frac{1}{2}v}z^\frac{1}{2}
e^{-\frac{1}{2}v}z^{-\frac{1}{2}}|\nabla z|\varphi^\frac{1}{2}dx\right)^2\\
&\leq 2\int_{B_{2r}(x_0) \cap \Omega} ze^v dx \iO e^{-v} \dfrac{|\nabla z|^2}{z}\varphi dx.
\end{align*}
On the other hand, since it holds from Lemma \ref{lem:30} that 
\begin{align*}
|\nabla \varphi^\frac{1}{2}| &= \dfrac{1}{2}|\varphi^{-\frac{1}{2}}\nabla \varphi| \leq \dfrac{A}{2}\varphi^{\frac{1}{2}-\frac{1}{n}}
\end{align*}
on $\{x \in \Omega; \varphi (x) \not= 0\}$, for 
sufficiently large $n$ we can handle the second term as follows:
\begin{align*}
2\left(\iO z|\nabla \varphi^\frac{1}{2}|dx\right)^2
&\leq \dfrac{A^2}{2}\left(\iO z\varphi^{\frac{1}{2}-\frac{1}{n}} dx\right)^2\\
&\leq \dfrac{A^2}{2}\|z\|_{L^1(B_{2r}(x_0)\cap \Omega)}^2.
\end{align*}
Finally using the Sobolev inequality \eqref{ine:sob} with $f = z\varphi^\frac{1}{2}$, we complete the proof.
\end{proof}

\subsection{Proof of Theorem \ref{th:1}}\label{subsec:2}

In this part, we will present a rigorous proof of Theorem \ref{th:1}. To begin with, we confirm that the grow-up point is the only origin provided that the initial data $(u_0,v_0)$ is a couple of radially symmetric functions.

\begin{proposition}\label{pro:rad0}
Let $\Omega = B_R \subset \R^2$ with some $R > 0$ and $(u_0,v_0)$ be a couple of radially symmetric functions satisfying \eqref{c}. Then $\mathcal{G} = \{0\}$.
\end{proposition}

\begin{proof}
Since the initial data $(u_0,v_0)$ is a couple of radially symmetric functions, the corresponding solution $(u,v)$ of \eqref{p} is
also radially symmetric. Suppose that there exists a grow-up point $x_0 \not= 0$. It follows from the definition of grow-up points that there exist $\{x_k\} \subset \overline{\Omega}$ and $\{t_k\} \subset (0,\infty)$ such that
\begin{align}
t_k \to \infty,\quad x_k \to x_0,\quad u(x_k, t_k) \to \infty\ \mathrm{as}\ k \to \infty.\label{eq:gro}
\end{align}
Here we notice that $v(x_k, t_k) \to \infty$ as $k \to \infty$. Indeed, suppose to the contrary that $v$ is uniformly bounded in time around $x_0$. This implies $u$ is also uniformly bounded in time around $x_0$ applying \cite[Theorem 10.1]{LSU1968} for the first equation of \eqref{p}, which contradicts \eqref{eq:gro}.
Now Proposition \ref{pro:322} yields that there exists $C$ independent of $t$ such that
\begin{align*}
v(x_k, t_k) \leq C(\|u_0\|_{L^1(\Omega)} + \|v_0\|_{W^{1,p}(\Omega)})|x_k|^{-\frac{2-p}{p}}
\end{align*}
with $p \in (1,2)$. Since the left-hand side diverges to $\infty$ and the right-hand side stays bounded as
$k \to \infty$, we get a contradiction. We thus conclude that the grow-up point is the only origin.
\end{proof}

The following lemma is the variant of $\varepsilon$-regularity theorem. 
There are several versions of $\varepsilon$-regularity estimates
for different types of chemotaxis systems: the Keller--Segel system on smooth bounded domain by Nagai--Senba--Suzuki \cite{NSS2000} and Senba--Suzuki \cite{SS2001}, and the fully parabolic Keller--Segel system with degenerate diffusion in $\R^2$ by Sugiyama \cite{SY2010}, and chemotaxis system with indirect signal production by the author \cite{YS2025}. However, these methods of the proof are not applicable to our system due to the motility function $e^{-v}$. For this reason, we establish alternative lemma for studying long-time behaviors of the solution to \eqref{p} in a localized area. 

\begin{lemma}\label{lem:3.5}
Let $x_0 \in \overline{\Omega}$, $0 < r \ll 1$, and $n$ be sufficiently large, and more $\varphi = \varphi_{(x_0,r,n)}$ be as in Lemma \ref{lem:30}. Assume that it holds that there exists some $\varepsilon > 0$ such that
for all $t \in (0,\infty)$
\begin{align*}
\int_{B_{2r}(x_0)} ue^v dx \leq \varepsilon.
\end{align*}
Then, there exists a positive constant $C$ independent of time such that for all $t \in (0,\infty)$,
\begin{align}
&E(t) + \dfrac{1}{4}\int_0^t e^{s-t}\iO e^{-v}\dfrac{|\nabla u|^2}{u}\varphi dxds\notag\\
&\hspace{0.5cm} + \dfrac{1}{4}\int_0^t e^{s-t} \iO |v_t|^2 \varphi dxds + \dfrac{1}{4}\int_0^t e^{s-t} \iO |\Delta v|^2 \varphi dxds\notag\\
&\hspace{0.5cm}\leq C,\label{eq:3.5555}
\end{align}
where 
\begin{align*}
E(t) = \iO (u\log u)\varphi dx + \iO |\nabla v|^2 \varphi dx + \iO v^2 \varphi dx.
\end{align*}
\end{lemma}

\begin{proof}
The key idea of the proof is to impose differentiation on $\varphi$ by integrating by parts. 
Let us examine the term $\iB (u\log u)\varphi dx$. By using the first equation of \eqref{p} and integrating by parts, it holds that 
\begin{align}
&\dfrac{d}{dt}\left(\iB (u\log u) \varphi dx\right)\notag\\ &= \iB u_t (\log u) \varphi dx + \iB u_t \varphi dx\notag\\
&= \iB \nabla \cdot (ue^{-v} \nabla (\log u - v)) (\log u)\varphi dx + \iB \nabla \cdot (ue^{-v}\nabla (\log u - v)) \varphi dx\notag\\
&= - \iB e^{-v}\nabla (\log u - v) \cdot (\nabla u) \varphi dx - \iB (u\log u) e^{-v} \nabla (\log u - v) \cdot \nabla \varphi dx\notag\\
&\hspace{0.5cm}-\iB ue^{-v}\nabla (\log u - v) \cdot \nabla \varphi dx.\notag
\end{align}
Here, we have
\begin{align*}
- \iB e^{-v}\nabla (\log u - v) \cdot (\nabla u) \varphi dx &= -\iB e^{-v}\dfrac{|\nabla u|^2}{u}\varphi dx + \iB e^{-v} (\nabla v \cdot \nabla u) \varphi dx
\end{align*}
and
\begin{align*}
&- \iB (u\log u) e^{-v} \nabla (\log u - v) \cdot \nabla \varphi dx\\
&= - \iB (\log u) e^{-v} \nabla u \cdot \nabla \varphi dx + \iB (u\log u) e^{-v} \nabla v \cdot \nabla \varphi dx
\end{align*}
and moreover
\begin{align*}
-\iB ue^{-v}\nabla (\log u - v) \cdot \nabla \varphi dx = -\iB e^{-v} \nabla u \cdot \nabla \varphi dx + \iB ue^{-v} \nabla v \cdot \nabla \varphi dx.
\end{align*}
Gathering these information, we get
\begin{align*}
&\dfrac{d}{dt}\iB (u\log u) \varphi dx + \iB e^{-v}\dfrac{|\nabla u|^2}{u}\varphi dx\\
&= \iB e^{-v} u (\log u + 1) \nabla v \cdot \nabla \varphi dx - \iB e^{-v}(\log u + 1)\nabla u \cdot \nabla \varphi dx + \iB e^{-v} (\nabla v \cdot \nabla u) \varphi dx.
\end{align*}
Noticing that
\begin{align*}
&-\iB e^{-v} (\log u +1) \nabla u \cdot \nabla \varphi dx\\
 &= \iB e^{-v} (\log u + 1) u \Delta \varphi dx -\iB u e^{-v} (\log u +1) \nabla v \cdot \nabla \varphi dx + \iB e^{-v} \nabla u \cdot \nabla \varphi dx,
\end{align*}
we establish
\begin{align}
&\dfrac{d}{dt}\iB (u\log u) \varphi dx + \iB e^{-v}\dfrac{|\nabla u|^2}{u}\varphi dx\notag\\
&= \iB e^{-v} (\log u + 1)u \Delta \varphi dx + \iB e^{-v} (\nabla v \cdot \nabla u) \varphi dx + \iB e^{-v} \nabla u \cdot \nabla \varphi dx.\label{eq:3333}
\end{align}
Moreover, Lemma \ref{lem:30} and $|u\log u| \leq u\log u + \frac{2}{e}$ imply
\begin{align}
\iB e^{-v}(1 + \log u ) u \Delta \varphi dx &\leq B\iB e^{-v} u \varphi^{1-\frac{2}{n}} dx + B\iB e^{-v}|u\log u| \varphi^{1-\frac{2}{n}} dx\notag\\
&\leq B \iB e^{-v}|u \log u | \varphi^{1-\frac{2}{n}}dx + B\iB u \varphi^{1-\frac{2}{n}}
dx\notag\\
&\leq B \iB e^{-v}(u\log u) \varphi^{1-\frac{2}{n}}dx + B \iB \dfrac{2}{e}e^{-v}\varphi^{1-\frac{2}{n}} dx + B\iB u \varphi^{1-\frac{2}{n}}dx.\notag
\end{align}
Now, we can find a positive constant $C$ satisfying
\begin{align}
B\left(u \log u + \dfrac{2}{e}\right) \leq u^{1 + \frac{1}{n}} + C.\label{ulogu}
\end{align}
Hence, it holds that
\begin{align}
\iB e^{-v}(1 + \log u ) u \Delta \varphi dx &\leq \iB e^{-v}u ^{1 + \frac{1}{n}} \varphi^{1 - \frac{2}{n}}dx
+\iO C\varphi^{1-\frac{2}{n}} dx 
 + B\iB u \varphi^{1-\frac{2}{n}} dx\notag\\
 &\leq \iO e^{-v}u ^{1 + \frac{1}{n}} \varphi^{1 - \frac{2}{n}}dx
+ B\iB u \varphi^{1-\frac{2}{n}} dx + C.\notag
\end{align}
We remark that there exists $ n_0 \in \N$ such that
\begin{align*}
\left(1 - \dfrac{2}{n}\right)\cdot \left( \dfrac{2n}{n+1}\right) = \dfrac{2(n-2)}{n + 1} \ge 1 \quad \mathrm{for\ all}\ n \ge n_0,
\end{align*}
and $\varphi^\alpha \leq \varphi$ for $\alpha \ge 1$ due to the properties of $\varphi$. Here, we may choose sufficiently large $ n \ge n_0$. 
Using the Young inequality and the H\"{o}lder inequality, we have for any $\varepsilon > 0$
\begin{align}
\iB e^{-v}(1 + \log u )u \Delta \varphi dx &\leq \left(\iB e^{-v}u^2 \varphi dx\right)^\frac{n+1}{2n} |\Omega|^\frac{n-1}{2n} + B\iB u \varphi^{1-\frac{2}{n}}dx + C\notag\\
&\leq \varepsilon \iB e^{-v}u^2 \varphi dx + C(\varepsilon)\label{eq:3.4.4}
\end{align}
due to the mass conservation law \eqref{mass:1}. We next deal with the term $ \iB e^{-v} \nabla u \cdot \nabla \varphi dx$. Thanks to Lemma \ref{lem:30} and sufficiently large $n$, and the mass conservation law \eqref{mass:1}, 
we obtain from the positivity of the solution and the Cauchy--Schwarz inequality, and the Young inequality that
\begin{align}
\iB e^{-v} \nabla u \cdot \nabla \varphi dx &\leq  A\iO e^{-v}|\nabla u|\varphi^{1-\frac{1}{n}} dx\notag\\
&\leq A\left(\iO e^{-v}\dfrac{|\nabla u|^2}{u}\varphi dx\right)^\frac{1}{2}\left(\iO e^{-v}u\varphi^{1-\frac{2}{n}} dx\right)^\frac{1}{2}\notag\\
&\leq A\|u_0\|_{L^1(\Omega}^\frac{1}{2}\left(\iO e^{-v}\dfrac{|\nabla u|^2}{u}\varphi dx\right)^\frac{1}{2}\notag\\
&\leq \dfrac{1}{2}\iO e^{-v}\dfrac{|\nabla u|^2}{u}\varphi dx + \dfrac{A^2\|u_0\|_{L^1(\Omega)}}{2}.
\label{eq:3.4.5}
\end{align}
As to the rest of the term on the right--hand side of \eqref{eq:3333}, by substituting 
the second equation of \eqref{p} and integrating by parts we obtain the following consequence from the positivity of the solution:
\begin{align*}
&\iB e^{-v}(\nabla v \cdot \nabla u) \varphi dx\\
&= -\iB e^{-v}u (\Delta v)\varphi dx - \iB e^{-v} u \nabla v \cdot \nabla \varphi dx + \iB e^{-v}u|\nabla v|^2 \varphi dx\\
&=- \iB e^{-v}u (v_t + v - u)\varphi dx - \iB e^{-v} u \nabla v \cdot \nabla \varphi dx + \iB e^{-v}u|\nabla v|^2 \varphi dx\\
&\leq - \iB e^{-v}u v_t dx + \iB e^{-v}u^2 \varphi dx - \iB e^{-v} u \nabla v \cdot \nabla \varphi dx + \iB e^{-v}u|\nabla v|^2 \varphi dx.
\end{align*}
Similar to the discussion in \eqref{eq:3.4.5}, we get from Lemma \ref{lem:30} that
\begin{align}
&\iB e^{-v}(\nabla v \cdot \nabla u) \varphi dx\notag\\
&\leq \varepsilon \iB e^{-v}u^2 \varphi dx + \dfrac{1}{4\varepsilon}\iB e^{-v}|v_t|^2 \varphi dx +  \iB e^{-v}u^2 \varphi dx\notag\\
&\hspace{0.5cm} +A \iB e^{-v}u |\nabla v|\varphi^{1-\frac{1}{n}}dx + \iB e^{-v} u|\nabla v|^2 \varphi dx \notag\\
&\leq (2\varepsilon + 1) \iB e^{-v}u^2 \varphi dx + \dfrac{1}{4\varepsilon}\iB e^{-v}|v_t|^2 \varphi dx + \dfrac{1}{4\varepsilon}\iB e^{-v}|\nabla v|^4 \varphi dx\notag\\
&\hspace{0.5cm}+ \varepsilon \iB e^{-v}u^2 \varphi dx + \dfrac{A^2}{4\varepsilon}\iB |\nabla v|^2 \varphi^{1-\frac{2}{n}} dx\notag\\
&= (3\varepsilon + 1)\iB e^{-v} u^2 \varphi dx + + \dfrac{1}{4\varepsilon}\iB |v_t|^2 \varphi dx + \dfrac{1}{4\varepsilon}\iB |\nabla v|^4 \varphi dx+ \dfrac{A^2}{4\varepsilon}\iB |\nabla v|^2 \varphi^{1-\frac{2}{n}} dx.\label{eq:3.4.6}
\end{align}
Thanks to \eqref{semi:1}, 
the term $\iB |\nabla v|^2\varphi^{1-\frac{2}{n}} dx$ on the right--hand side of the above inequality can be evaluated as
\begin{align}
&\dfrac{A^2}{4\varepsilon}\iB |\nabla v|^2 \varphi^{1-\frac{2}{n}} dx\notag\\
&= \dfrac{A^2}{4\varepsilon}\iB (\nabla v \cdot \nabla v) \varphi^{1-\frac{2}{n}} dx\notag\\
&=- \dfrac{A^2}{4\varepsilon}\iB v(\Delta v) \varphi^{1-\frac{2}{n}} dx - \dfrac{A^2}{4\varepsilon}\iB v \nabla v \cdot  \nabla \varphi^{1-\frac{2}{n}} dx\notag\\
&\leq \dfrac{1}{4\varepsilon}\iB |\Delta v|^2 \varphi dx + \dfrac{A^4}{16\varepsilon}\iB v^2 \varphi^{1-\frac{4}{n}} dx - \dfrac{A^2}{8\varepsilon} \iB \nabla v^2 \cdot \nabla \varphi^{1-\frac{2}{n}} dx\notag\\
&\leq \dfrac{1}{4\varepsilon}\iB |\Delta v|^2 \varphi dx + C(\varepsilon) + \dfrac{A^2}{8\varepsilon}\|\Delta \varphi^{1-\frac{2}{n}}\|_{L^\infty(\Omega)}\iO v^2 dx\notag\\
&\leq \dfrac{1}{4\varepsilon}\iB |\Delta v|^2 \varphi dx + C(\varepsilon)\label{eq:3.4.7}
\end{align}
for sufficiently large $n$. Since $(a + b)^4 \leq 8(a^4 + b^4)$ for $a,b \ge 0$, we obtain from \eqref{semi:1} that
\begin{align}
\dfrac{1}{4\varepsilon}\iB |\nabla v|^4 \varphi dx &= \dfrac{1}{4\varepsilon}\iB |\nabla (v\varphi^\frac{1}{4}) - v\nabla (\varphi^\frac{1}{4})|^4 dx\notag\\
&\leq \dfrac{2}{\varepsilon}\iB |\nabla (v\varphi^\frac{1}{4})|^4 dx + \dfrac{2\|\nabla \varphi^\frac{1}{4}\|_{L^\infty(\Omega)}^4}{\varepsilon}\iB v^4  dx\notag\\
&\leq \dfrac{2}{\varepsilon}\iB |\nabla (v\varphi^\frac{1}{4})|^4 dx + C(\varepsilon).\label{eq:3.4.8}
\end{align}
Combining \eqref{eq:3.4.4}-\eqref{eq:3.4.8} with \eqref{eq:3333} implies that
\begin{align}
&\dfrac{d}{dt}\iB (u\log u)\varphi dx + \dfrac{1}{2}\iB e^{-v}\dfrac{|\nabla u|^2}{u}\varphi dx\notag\\
&\leq (4\varepsilon + 1)\iB e^{-v}u^2 \varphi dx + \dfrac{1}{4\varepsilon}\iB |v_t|^2 \varphi dx\notag\\
&\hspace{0.5cm} + \dfrac{1}{4\varepsilon}\iB |\Delta v|^2 \varphi dx + \dfrac{2}{\varepsilon}\iB |\nabla (v\varphi^\frac{1}{4})|^4 dx + C(\varepsilon).\label{eq:3.4.9}
\end{align}
Next, multiplying the second equation of \eqref{p} by $-(\Delta v)\varphi$ and integrating by parts, we have 
\begin{align*}
\dfrac{1}{2}\dfrac{d}{dt}\iB |\nabla v|^2 \varphi dx &+\iB |\Delta v|^2 \varphi dx + 
\iB |\nabla v|^2 \varphi dx \\
&= -\iB u(\Delta v) \varphi dx - \iB v_t\nabla v \cdot \nabla \varphi dx - \iB v\nabla v \cdot \nabla \varphi dx.
\end{align*}
Furthermore, integrating by parts after multiplying the second equation of \eqref{p} by $v_t \varphi$ and $v \varphi$ respectively, we similarly get 
\begin{align*}
\iB |v_t|^2\varphi dx &+ \dfrac{1}{2}\dfrac{d}{dt}\iB |\nabla v |^2\varphi dx + \dfrac{1}{2}\dfrac{d}{dt}\iB v^2 \varphi dx = \iB u v_t \varphi dx - \iB v_t\nabla v \cdot \nabla \varphi dx
\end{align*}
and 
\begin{align*}
\dfrac{1}{2}\dfrac{d}{dt}\iB v^2 \varphi dx + \iB |\nabla v|^2 \varphi dx + \iB v^2 \varphi dx = \iB uv\varphi dx - \iB v \nabla v \cdot \nabla \varphi dx.
\end{align*}
By multiplying each of the above inequalities by $\lambda$, which will be chosen later, and then collecting them, it holds from the H\"{o}lder inequality and Lemma \ref{lem:30} that
\begin{align}
&\lambda\dfrac{d}{dt}\iB |\nabla v|^2 \varphi dx + \lambda\dfrac{d}{dt}\iB v^2\varphi dx + \lambda \iO |\Delta v|^2 \varphi dx + 2\lambda\iB |\nabla v|^2\varphi dx\notag\\
 &\hspace{0.5cm}+ \lambda \iO v^2 \varphi dx +\lambda \iO |v_t|^2 \varphi dx\notag\\
&= -\lambda \iO u(\Delta v)\varphi dx - 2\lambda \iO v_t \nabla v \cdot \nabla \varphi dx -2\lambda \iO v \nabla v \cdot \nabla \varphi dx\notag\\
&\hspace{0.5cm}+ \lambda \iO uv_t \varphi dx + \lambda \iO uv\varphi dx\notag\\
&=: I + II + III + IV + V.\label{eq:3.4.10}
\end{align}
Here, integrating by parts and  $v \in L^\infty(0,\infty; L^2(\Omega))$ imply
\begin{align}
III &= -2\lambda \iO v \nabla v \cdot \nabla \varphi dx\notag\\
&= \lambda \iO v^2 \Delta \varphi dx\notag\\
&\leq C(\lambda).\label{eq:3.4.11}
\end{align}
With respect to $I$, $IV$, and $V$ respectively, we see from the H\"{o}lder inequality that
\begin{align}
I &\leq \lambda\iB u^2 \varphi dx + \dfrac{\lambda}{4}\iO |\Delta v|^2 \varphi dx\label{eq:3.4.12}
\end{align}
and 
\begin{align}
IV \leq \lambda\iB u^2 \varphi dx + \dfrac{\lambda}{4}\iO |v_t|^2 \varphi dx\label{eq:3.4.13}
\end{align}
and moreover from \eqref{semi:1}
\begin{align}
V &\leq \lambda\iB u^2 \varphi dx + \dfrac{\lambda}{4}\iO v^2 \varphi dx\notag\\
&\leq \lambda\iB u^2 \varphi dx + C(\lambda).\label{eq:3.4.14}
\end{align}
By using the H\"{o}lder inequality and Lemma \ref{lem:30}, the term $II$ can be calculated as follows:
\begin{align*}
II &\leq \dfrac{\lambda}{4}\iO |v_t|^2 \varphi dx + 4A^2\lambda \iO |\nabla v|^2 \varphi^{1-\frac{2}{n}} dx.
\end{align*}
Similar to \eqref{eq:3.4.7}, the second term on the right--hand side of the above inequality can be rewritten as
\begin{align}
4A^2\lambda \iO |\nabla v|^2 \varphi^{1-\frac{2}{n}}dx \leq \dfrac{\lambda}{4}\iO |\Delta v|^2 \varphi dx + C(\lambda).\label{eq:3.4.15}
\end{align}
Gathering \eqref{eq:3.4.10}-\eqref{eq:3.4.15} yields that
\begin{align}
&\lambda\dfrac{d}{dt}\iB |\nabla v|^2 \varphi dx + \lambda\dfrac{d}{dt}\iB v^2\varphi dx + \dfrac{\lambda}{2}\iO |\Delta v|^2 \varphi dx + 2\lambda\iB |\nabla v|^2\varphi dx\notag\\
&\hspace{0.5cm}+ \lambda\iO v^2 \varphi dx + \dfrac{\lambda}{2}\iO |v_t|^2 \varphi dx\notag\\
& \leq 3\lambda\iO u^2\varphi dx + C(\lambda).\label{eq:3.4.16}
\end{align}
By collecting \eqref{eq:3.4.9} and \eqref{eq:3.4.16}, and choosing $\varepsilon = \lambda = 1$, we have
\begin{align*}
&\dfrac{d}{dt}\left(\iO (u\log u) \varphi dx + \iO |\nabla v|^2 \varphi dx + \iO v^2 \varphi dx\right) + 2\iO |\nabla v|^2 \varphi dx + \iO v^2 \varphi dx\notag\\
 &\hspace{0.5cm} + \dfrac{1}{4}\iO |\Delta v|^2 \varphi dx + \dfrac{1}{4}\iO |v_t|^2 \varphi dx + \dfrac{1}{2}\iO e^{-v}\dfrac{|\nabla u|^2}{u}\varphi dx\notag\\
 &\leq 5\iO e^{-v}u^2\varphi dx + 2 \iO |\nabla (v\varphi^\frac{1}{4})|^4 dx + 3\iO u^2\varphi dx + C\notag\\
 &\leq 8\iO u^2 \varphi dx + 2 \iO |\nabla (v\varphi^\frac{1}{4})|^4 dx + C.
\end{align*}
We add $\iO (u\log u)\varphi dx$ to the above inequality to get
\begin{align*}
&\dfrac{d}{dt}\left(\iO (u\log u) \varphi dx + \iO |\nabla v|^2 \varphi dx + \iO v^2 \varphi dx\right) + \iO (u\log u)\varphi dx + 2\iO |\nabla v|^2 \varphi dx\\
&\hspace{0.5cm}+ \iO v^2 \varphi dx + \dfrac{1}{4}\iO |\Delta v|^2 \varphi dx + \dfrac{1}{4}\iO |v_t|^2 \varphi dx + \dfrac{1}{2}\iO e^{-v}\dfrac{|\nabla u|^2}{u}\varphi dx\\
&\leq  8\iO u^2 \varphi dx + 2 \iO |\nabla (v\varphi^\frac{1}{4})|^4 dx + C + \iO (u\log u)\varphi dx
\end{align*}
and thus we obtain from the combination of \eqref{ulogu} and the H\"{o}lder inequality that
\begin{align*}
&\dfrac{d}{dt}\left(\iO (u\log u) \varphi dx + \iO |\nabla v|^2 \varphi dx + \iO v^2 \varphi dx\right) + \iO (u\log u)\varphi dx + 2\iO |\nabla v|^2 \varphi dx\\
&\hspace{0.5cm}+ \iO v^2 \varphi dx + \dfrac{1}{4}\iO |\Delta v|^2 \varphi dx + \dfrac{1}{4}\iO |v_t|^2 \varphi dx + \dfrac{1}{2}\iO e^{-v}\dfrac{|\nabla u|^2}{u}\varphi dx\\
&\leq 9\iO u^2 \varphi dx +  2 \iO |\nabla (v\varphi^\frac{1}{4})|^4 dx + C.
\end{align*}
Introducing
\begin{align*}
E(t) := \iB (u\log u) \varphi dx + \iB |\nabla v|^2 \varphi dx + \iB v^2 \varphi dx,
\end{align*}
we confirm that for all $t \in (0,\infty)$,
\begin{align*}
&\dfrac{d}{dt}E(t) + E(t) + \dfrac{1}{2}\iB e^{-v}\dfrac{|\nabla u|^2}{u}\varphi dx + \dfrac{1}{4}\iB |v_t|^2 \varphi dx + \dfrac{1}{4}\iB |\Delta v|^2 \varphi dx\\
&\leq 9\iO u^2 \varphi dx + 2 \iO |\nabla (v\varphi^\frac{1}{4})|^4 dx + C
\end{align*}
with a positive constant $C$. Multiplying the above inequality by $e^t$ and integrating over $(0,t)$ for all $t \in (0,\infty)$, we infer that
\begin{align}
&E(t)e^t + \dfrac{1}{2}\int_0^t e^s\iO e^{-v} \dfrac{|\nabla u|^2}{u}\varphi dxds + \dfrac{1}{4}\int_0^t e^s \iO |v_t|^2 \varphi dxds + \dfrac{1}{4}\int_0^t e^s \iO |\Delta v|^2 \varphi dxds\notag\\
&\leq E(0) + 9\int_0^t e^s\iO u^2 \varphi dx ds + 2\int_0^t e^s \iO |\nabla (v\varphi^\frac{1}{4})|^4 dxds + Ce^t.\label{eq:3.4.17}
\end{align}
Applying \cite[Lemma 2.8]{FS2016}, which is the variant form of maximal regularity for parabolic equations, to $2\int_0^t e^s \iO |\nabla (u\varphi^\frac{1}{4})|^4 dxds$, we get
\begin{align}
2\int_0^t e^s \iO |\nabla (v\varphi^\frac{1}{4})|^4 dxds \leq D_0 \int_0^t e^s \iO u^2 \varphi dxds + C(1 + e^t),\label{eq:3.4.18}
\end{align}
where $D_0$ is dependent on the Sobolev constant, the constant by the elliptic regularity and by the maximal regularity for parabolic equations, and $\|u_0\|_{L^1(\Omega)}$, and additionally $C$ is a positive constant. Therefore \eqref{eq:3.4.17} and \eqref{eq:3.4.18} yield that for all $t \in (0,\infty)$,
\begin{align}
&E(t)e^t + \dfrac{1}{2}\int_0^t e^s \iO e^{-v}\dfrac{|\nabla u|^2}{u}\varphi dx ds+ \dfrac{1}{4}\int_0^t e^s \iO |v_t|^2 \varphi dxds + \dfrac{1}{4}\int_0^t e^s \iO |\Delta v|^2 \varphi dxds\notag\\
&\leq (9 + D_0)\int_0^t e^s \iO u^2 \varphi dxds + C(1 + e^t).\label{eq:3.4.19}
\end{align}
If we have the assumption that 
\begin{align*}
\sup_{t > 0} \int_{B_{2r}} ue^v dx \leq \dfrac{1}{8(9+D_0)K_{\mathrm{Sob}}^2},
\end{align*}
we infer from Proposition \ref{pro:ine} that
\begin{align}
&(9 + D_0)\int_0^t e^s\iO u^2 \varphi dxds\notag\\
&\leq \dfrac{1}{4}\int_0^t e^s\iO e^{-v}\dfrac{|\nabla u|^2}{u}\varphi dxds + (9 + D_0)K^2_{\mathrm{Sob}}\left(\dfrac{A^2}{2} + 1\right)\|u_0\|_{L^1(\Omega)}^2\int_0^t e^s ds.\label{eq:3.4.20}
\end{align}
Hence we obtain that for all $t \in (0,\infty)$,
\begin{align*}
&E(t)e^t + \dfrac{1}{4}\int_0^t e^s\iO e^{-v}\dfrac{|\nabla u|^2}{u}\varphi dxds + \dfrac{1}{4}\int_0^t e^s \iO |v_t|^2 \varphi dxds + \dfrac{1}{4}\int_0^t e^s \iO |\Delta v|^2 \varphi dxds\notag\\
&\leq C(1 + e^t)
\end{align*}
and dividing the above inequality by $e^t$, we have
\begin{align*}
&E(t) + \dfrac{1}{4}\int_0^t e^{s-t}\iO e^{-v}\dfrac{|\nabla u|^2}{u}\varphi dxds + \dfrac{1}{4}\int_0^t e^{s-t} \iO |v_t|^2 \varphi dxds + \dfrac{1}{4}\int_0^t e^{s-t} \iO |\Delta v|^2 \varphi dxds\notag\\
&\leq C.
\end{align*}
This is the desired inequality.
\end{proof}

We shall derive some stability on the anulus.

\begin{lemma}\label{lem:bdd1}
Let $0 < r \ll 1$ and $(u,v)$ be a radially symmetric solution of \eqref{p}. Then there exists a positive constant $C(r)$ such that for all $t \in (0,\infty)$,
\begin{align}
\int_0^t e^{s-t}\int_{B_{2r} \setminus B_r} u^2 dx ds \leq C(r).\label{eq:3.4.200}
\end{align}
\end{lemma}

\begin{proof}
Let us give $r \in (0,\frac{R}{4})$ and a sufficiently large $N \in \N$. Then, for any $\eta \in [r,2r]$, we can take $\{x_i\}_{i = 1}^N \subset \overline{\Omega}$ such that
\begin{align}
|x_i| = \eta,\quad B_{\frac{r}{N}}(x_i) \cap B_{\frac{r}{N}}(x_j) = \emptyset,\ i \not= j.\label{eq:3.4.21}
\end{align}
Since the solution of \eqref{p} is a radially symmetric function, we have for $i, j \in [1,N]$,
\begin{align*}
\int_{B_{\frac{r}{N}}(x_i)} u e^v dx = \int_{B_{\frac{r}{N}}(x_j)} u e^v dx.
\end{align*}
Hence we infer from the above equality and \eqref{eq:3.4.21} that
\begin{align*}
N\left(\sup_{t > 0} \int_{B_{\frac{r}{N}}(x_i)} u e^v dx\right) &= \sum_{i = 1}^N \sup_{t > 0} \int_{B_{\frac{r}{N}}(x_i)} u e^v dx\\
&\leq \sup_{t > 0} \int_{B_R(0) \setminus B_{\frac{r}{2}}(0)} ue^v dx.
\end{align*}
A radial symmetry and Proposition \ref{pro:322} imply
\begin{align*}
N\left(\sup_{t > 0} \int_{B_{\frac{r}{N}}(x_i)} u e^v dx\right) &\leq 2\pi \sup_{t >0} \int_{\frac{r}{2}}^R \rho u e^v d\rho\\
&\leq 2\pi \sup_{t >0} \int_{\frac{r}{2}}^R \rho u e^{|||v|||_\kappa\rho^{-\kappa}} d\rho\\
&\leq e^{|||v|||_\kappa (\frac{r}{2})^{-\kappa}} \|u_0\|_{L^1(\Omega)}
\end{align*}
and thus
\begin{align*}
\sup_{t >0} \int_{B_{\frac{r}{N}}(x_i)} u e^v dx \leq \dfrac{e^{|||v|||_\kappa (\frac{r}{2})^{-\kappa}} \|u_0\|_{L^1(\Omega)}}{N}.
\end{align*}
Choosing sufficiently large $N$ such that
\begin{align*}
\dfrac{e^{|||v|||_\kappa (\frac{r}{2})^{-\kappa}} \|u_0\|_{L^1(\Omega)}}{N} \leq \varepsilon,
\end{align*}
where $\varepsilon$ is the constant in Lemma \ref{lem:3.5}, 
we can apply Lemma \ref{lem:3.5} for $x_0 \in \overline{\Omega}$ with $|x_0| = |x_i| = \eta$. As a result,  we obtain  \eqref{eq:3.5555}. Thanks to the above inequality, Proposition \ref{pro:ine} implies that there exists a positive constant $C$ independent of time such that for all $t \in (0,\infty)$,
\begin{align*}
\int_0^t e^{s-t} \int_{B_{\frac{r}{2N}}(x_0)} u^2 dxds \leq C.
\end{align*}
The compactness of $B_{2r} \setminus B_r$ allows us to take a finite set $\{x_i\}_{i = 1}^m$ such that
\begin{align*}
B_{2r} \setminus B_r \subset \bigcup_{i = 1}^m B_{\frac{r}{2N}}(x_i).
\end{align*}
For this reason, we conclude that for all $t \in (0,\infty)$,
\begin{align*}
\int_0^t e^{s-t} \int_{B_{2r} \setminus B_r} u^2 dx ds \leq C.
\end{align*}
\end{proof}

From now on, we always assume the grow-up point is the origin thanks to Proposition \ref{pro:rad0}. Additionally, we will write $B_r$ instead of $B_r(0)$ for $0< r < 1$ to shorten.

The next proposition gives rise to Theorem \ref{th:1}. The principal significance of the proposition is that the second component $v$, notably $\int_\Omega e^{v} dx$ may determine the grow-up of solutions. This essentially differs from that of the Keller--Segel system (see \cite{NSS2000, SS2001}) and some chemotaxis systems (see \cite{YS2025}). The key idea is to obtain $H^1$-regularity of the auxiliary function $w = (I-\Delta)^{-1}u$.

\begin{proposition}\label{prop:3.11}
Let $a > 1$, $0 < r \ll 1$, and $n$ be sufficiently large, and more $\varphi = \varphi_{(0,r,n)}$ be as in Lemma \ref{lem:30}. Assume that $(u,v)$ be a radially symmetric solution of \eqref{p} in $\Omega \times (0,\infty)$. Then, $(u,v)$ satisfies
\begin{align}
\limsup_{t \to \infty}\int_{\Omega} e^{av} \varphi dx = \infty\label{ite:0}
\end{align}
if and only if the origin is the grow-up point.
\end{proposition}

\begin{proof}
We assume \eqref{ite:0}.
Since we deduce from Lemma \ref{lem:30} that
\begin{align*}
\iO e^{av}\varphi dx &\leq e^{a\|v(t)\|_{L^\infty(\Omega)}}\iO \varphi dx\\
&\leq  e^{a\|v(t)\|_{L^\infty(\Omega)}}|\Omega|,
\end{align*}
the assertion \eqref{ite:0} yields
\begin{align*}
\limsup_{t \to \infty}\|v(t)\|_{L^\infty(B_{r})} = \infty.
\end{align*}
Suppose that the conclusion is false, that is to say, the origin is not grow-up point. Then, the definition of grow-up solutions yields that
\begin{align*}
\limsup_{t \to \infty}\|u(t)\|_{L^\infty(B_r)} < \infty.
\end{align*}
Here \cite[(STEP 3) in Proposition 3.7]{YS2025} contributes to achieving that
\begin{align*}
\sup_{t >0}\|v(t)\|_{L^\infty(B_r)} < \infty,
\end{align*}
which is the contradiction. Thus it holds that the origin is the grow up point.

We will show the converse statement by contradiction. 
Hence, we have the assumption that
\begin{align}
\limsup_{t \to \infty}\iO e^{av}\varphi dx < \infty.\label{ite:1}
\end{align}
First, we will prove that
\begin{align}
\sup_{t >0}\iO |\nabla w|^2 \varphi dx< \infty,\quad \sup_{t > 0}\int_0^t e^{s-t}\int_{B_r}e^{-v}u^2 dx ds < \infty.\label{ite:2}
\end{align}
Multiplying the first equation \eqref{p} by $w\varphi$, we use the integration by parts to obtain
\begin{align*}
\iO u_t w\varphi dx &= \iO \Delta (ue^{-v})w\varphi dx\\
&= \iO ue^{-v}(\Delta w) \varphi dx + 2\iO ue^{-v}\nabla w \cdot \nabla \varphi dx + \iO ue^{-v}w \Delta \varphi dx
\end{align*}
and due to the definition of $w$,
\begin{align*}
\iO u_t w \varphi dx = \iO u e^{-v} w\varphi dx - \iO e^{-v} u^2 \varphi dx + 2\iO ue^{-v}\nabla w \cdot \nabla \varphi dx + \iO ue^{-v}w \Delta \varphi dx.
\end{align*}
Here, the left-hand side above is rewritten by the definition of $w$ and the integration by parts as follows:
\begin{align*}
\iO u_t w \varphi dx &= \iO (-\Delta w_t + w_t) w \varphi dx\\
&= \dfrac{1}{2}\dfrac{d}{dt}\iO |\nabla w|^2 \varphi dx + \dfrac{1}{2}\dfrac{d}{dt}\iO w^2 \varphi dx + \iO w \nabla w_t \cdot \nabla \varphi dx.
\end{align*}
Consequently, we arrive at
\begin{align}
&\dfrac{1}{2}\dfrac{d}{dt}\iO |\nabla w|^2 \varphi dx + \dfrac{1}{2}\dfrac{d}{dt}\iO w^2 \varphi dx + \iO e^{-v}u^2 \varphi dx\notag\\
&= \iO ue^{-v} w\varphi dx + 2\iO ue^{-v}\nabla w \cdot \nabla \varphi dx + \iO ue^{-v}w \Delta \varphi dx - \iO w\nabla w_t \cdot \nabla \varphi dx.\label{ite:3}
\end{align}
Applying Lemma \ref{lem:h2} with $z = w, f = u$, and \eqref{mass:1} yield that for any $p \in [1,\infty)$,
\begin{align*}
\|w(t)\|_{L^p(\Omega)} \leq C\|u_0\|_{L^1(\Omega)},
\end{align*}
we get from the H\"older inequality and the Young inequality that
\begin{align}
\iO ue^{-v}w \varphi dx &\leq \dfrac{1}{8}\iO e^{-v}u^2 \varphi dx + 2\iO e^{-v}w^2\varphi dx\notag\\
&\leq \dfrac{1}{8}\iO e^{-v} u^2 \varphi dx + 2\iO w^2 dx\notag\\
&\leq \dfrac{1}{8}\iO e^{-v} u^2 \varphi dx + C.\label{ite:4}
\end{align}
Similar to the above argument and from Lemma \ref{lem:30}, we obtain for sufficiently large $n$,
\begin{align}
\iO ue^{-v} w \Delta \varphi dx &\leq B\iO ue^{-v}w \varphi^{1-\frac{2}{n}} dx\notag\\
&\leq \dfrac{1}{8}\iO u^2 e^{-v}\varphi dx + 2B^2 \iO e^{-v}w^2\varphi^{1-\frac{4}{n}} dx\notag\\
&\leq \dfrac{1}{8}\iO e^{-v}u^2\varphi dx + 2B^2 \iO w^2 dx\notag\\
&\leq \dfrac{1}{8}\iO e^{-v}u^2\varphi dx + C.\label{ite:5}
\end{align}
Thanks to Lemma \ref{lem:30}, we derive from the H\"older inequality and the Young inequality that
\begin{align}
2\iO ue^{-v}\nabla w \cdot \nabla \varphi dx \leq \dfrac{1}{8}\iO u^2 e^{-v}\varphi dx + 8A^2 \iO e^{-v}|\nabla w |^2 \varphi^{1-\frac{2}{n}} dx.\label{ite:6}
\end{align}
By the definition of $w$ and the integration by parts we obtain
\begin{align*}
8A^2\iO e^{-v}|\nabla w|^2 \varphi^{1-\frac{2}{n}} dx &\leq 8A^2\iO (\nabla w \cdot \nabla w)\varphi^{1-\frac{2}{n}} dx\\
&= -8A^2\iO w(\Delta w)\varphi^{1-\frac{2}{n}} dx - 8A^2\iO w \nabla w \cdot \nabla \varphi^{1-\frac{2}{n}}dx\\
&\leq 8A^2 \iO uw \varphi^{1-\frac{2}{n}} dx - 4A^2 \iO \nabla w^2 \cdot \nabla \varphi^{1-\frac{2}{n}} dx\\
&=: I + II.
\end{align*}
With respect to the term $I$, we use the H\"older inequality and the Young inequality to get
\begin{align*}
I &= 8A^2 \iO uw\varphi^{1-\frac{2}{n}} dx\\
&\leq \dfrac{1}{8}\iO e^{-v}u^2 \varphi dx + 128A^4 \iO e^{v} w^2 \varphi^{1-\frac{4}{n}} dx\\
&\leq \dfrac{1}{8}\iO e^{-v}u^2 \varphi dx + 128A^4 \left(\iO e^{av}\varphi dx\right)^\frac{1}{a}\left(\iO w^{\frac{2a}{a-1}} \varphi^{1-\frac{4a}{n(a-1)}} dx\right)^\frac{a-1}{a}
\end{align*}
and owing to the assumption \eqref{ite:1} and \eqref{eq:h3}, and more sufficiently large $n$, we infer that
\begin{align*}
I \leq \dfrac{1}{8}\iO e^{-v}u^2\varphi dx + C.
\end{align*}
With respect to the term $II$, by integrating by parts, Lemma \ref{lem:30} implies that
\begin{align*}
II &= 4A^2 \iO w^2 \Delta \varphi^{1-\frac{2}{n}} dx\\
&\leq 4A^2 \|\Delta \varphi^{1-\frac{2}{n}}\|_{L^\infty(\Omega)}\iO w^2 dx\\
&\leq C.
\end{align*}
Hence combing these estimates gives that
\begin{align}
8A^2\iO e^{-v}|\nabla w|^2 \varphi dx \leq \dfrac{1}{8}\iO e^{-v}u^2 \varphi dx + C.\label{ite:7}
\end{align}
Now collecting \eqref{ite:3}-\eqref{ite:7}, we establish that
\begin{align}
\dfrac{1}{2}\dfrac{d}{dt}\iO |\nabla w|^2 \varphi dx + \dfrac{1}{2}\dfrac{d}{dt}\iO w^2 \varphi dx + \dfrac{1}{2}\iO e^{-v}u^2 \varphi dx \leq -\iO w \nabla w_t \cdot \nabla \varphi dx + C.\label{ite:8}
\end{align}
We evaluate the remainder term on the right-hand side of the above inequality. Integrating by parts implies
\begin{align}
-\iO w \nabla w_t \cdot \nabla \varphi dx &= \iO w w_t \Delta \varphi dx + \iO w_t \nabla w \cdot \nabla \varphi dx\notag\\
&=: III + IV.\label{ite:10}
\end{align}
Since we derive from the first equation of \eqref{p} and the definition of $w$ that
\begin{align}
w_t &= (I - \Delta)^{-1}u_t\notag\\
&= (I - \Delta)^{-1}\Delta(ue^{-v})\notag\\
&= (I-\Delta)^{-1}[ue^{-v}] - ue^{-v},\label{ite:10.5}
\end{align}
the term $III$ can be rewritten as
\begin{align}
III = \iO w (I-\Delta)^{-1}[ue^{-v}]\Delta \varphi dx - \iO w (ue^{-v}) \Delta \varphi dx.\label{ite:11}
\end{align}
With respect to the first term on the right-hand side of the above equality, we get from Lemma \ref{lem:30}, \eqref{eq:h3}, and the mass conservation law \eqref{mass:1} that
\begin{align}
\iO w (I-\Delta)^{-1}[ue^{-v}]\Delta \varphi dx &\leq B\iO w |(I-\Delta)^{-1}[ue^{-v}]|\varphi^{1-\frac{2}{n}} dx\notag\\
&\leq B\|(I-\Delta)^{-1}[ue^{-v}]\|_{L^\infty(B_{2r})}\iO w \varphi^{1-\frac{2}{n}} dx\notag\\
&\leq B \|(I-\Delta)^{-1}[ue^{-v}]\|_{L^\infty(B_{2r})}C\|u_0\|_{L^1(\Omega)}.\notag
\end{align}
Since the Sobolev inequality and the elliptic regularity theorem allow us to calculate 
\begin{align*}
\|(I-\Delta)^{-1}[ue^{-v}]\|_{L^\infty(B_{2r})} \leq C_\mathrm{Sob}C_\mathrm{Ell}\left(\int_{B_{2r}}e^{-2v}u^2 dx \right)^\frac{1}{2},
\end{align*}
using the Young inequality, we have
\begin{align}
\iO w (I-\Delta)^{-1}[ue^{-v}]\Delta \varphi dx &\leq \dfrac{1}{16}\int_{B_{2r}} e^{-v}u^2 dx + 4B^2C_\mathrm{Sob}^2C_\mathrm{Ell}^2C^2\|u_0\|_{L^1(\Omega)}^2.\label{ite:12}
\end{align}
According to Lemma \ref{lem:30}, the H\"older inequality, the Young inequality, and \eqref{eq:h3},
the second term on the right-hand side of the equality \eqref{ite:11} can be estimated as 
\begin{align}
-\iO w(ue^{-v})\Delta \varphi dx &\leq B\iO w (ue^{-v}) \varphi^{1-\frac{2}{n}} dx\notag\\
&\leq B\left(\iO e^{-v}u^2 \varphi dx\right)^\frac{1}{2}\left(\iO e^{-v}w^2 \varphi^{1-\frac{4}{n}} dx\right)^\frac{1}{2}\notag\\
&\leq \dfrac{1}{16}\iO e^{-v}u^2 \varphi dx + 4B^2 \iO w^2 \varphi^{1-\frac{4}{n}} dx\notag\\
&\leq \dfrac{1}{16}\iO e^{-v}u^2 \varphi dx + C.\label{ite:13}
\end{align}
Combing \eqref{ite:11}-\eqref{ite:13} implies that
\begin{align}
III \leq \dfrac{1}{16}\int_{B_{2r}}e^{-v}u^2 dx + \dfrac{1}{16}\iO e^{-v}u^2 \varphi dx + C.\label{ite:14}
\end{align}
With respect to the term $IV$, by using \eqref{ite:10.5} once again, we deduce from Lemma \ref{lem:30} that
\begin{align*}
IV &= \iO ((I-\Delta)^{-1}[ue^{-v}]) \nabla w \cdot \nabla \varphi dx - \iO ue^{-v}\nabla w \cdot \nabla \varphi dx\notag\\
&\leq A \|(I-\Delta)^{-1}[ue^{-v}]\|_{L^\infty(B_{2r})}\iO |\nabla w|\varphi^{1-\frac{1}{n}} dx
+ A\iO ue^{-v}|\nabla w|\varphi^{1-\frac{1}{n}} dx.\notag
\end{align*}
Here, since the application of Lemma \ref{lem:h1} with $z = w$, $f = u$, and $p = 1$ and furthermore the mass conservation law \eqref{mass:1} ensure that
\begin{align*}
\|w(t)\|_{W^{1,1}(\Omega)} \leq C\|u_0\|_{L^1(\Omega)},
\end{align*}
where $C$ is a positive constant, we can now proceed analogously to the estimate \eqref{ite:11}-\eqref{ite:13} as follows:
\begin{align}
IV \leq \dfrac{1}{16}\int_{B_{2r}} e^{-v}u^2 dx + A\iO ue^{-v}|\nabla w| \varphi^{1-\frac{1}{n}} dx + C.\label{ite:15}
\end{align}
Moreover, by using the H\"older inequality and the Young inequality, we get
\begin{align}
A\iO ue^{-v}|\nabla w| \varphi^{1-\frac{1}{n}} dx &\leq \dfrac{1}{16}\int_{B_{2r}}e^{-v}u^2 dx + 4A^2 \iO |\nabla w|^2 \varphi dx\label{ite:16}
\end{align}
due to the fact $\varphi^\alpha \leq \varphi$ for $\alpha \ge 1$. 
Gathering \eqref{ite:10} and \eqref{ite:14}-\eqref{ite:16} implies
\begin{align*}
-\iO w \nabla w_t \cdot \nabla \varphi dx &\leq \dfrac{3}{16}\int_{B_{2r}} e^{-v}u^2 dx + \dfrac{1}{16}\iO e^{-v}u^2 \varphi dx + 4A^2\iO |\nabla w|^2 \varphi dx + C\\
&\leq \dfrac{1}{4}\int_{B_{2r}} e^{-v}u^2 dx + 4A^2\iO |\nabla w|^2 \varphi dx + C
\end{align*}
and also combining the above estimate with \eqref{ite:8} bring about
\begin{align}
&\dfrac{1}{2}\dfrac{d}{dt}\iO |\nabla w|^2 \varphi dx + \dfrac{1}{2}\dfrac{d}{dt}\iO w^2 \varphi dx + \dfrac{1}{2}\iO e^{-v}u^2 \varphi dx\notag\\
&\leq \dfrac{1}{4}\int_{B_{2r}} e^{-v}u^2 dx + 4A^2\iO |\nabla w|^2 \varphi dx + C.\label{ite:17}
\end{align}
Multiplying $-\Delta w + w = u$ by $w\varphi$ and integrating by parts in order to obtain time-independent estimates of $\|w(t)\|_{H^1(\Omega)}$, 
we have
\begin{align*}
\iO |\nabla w|^2 \varphi dx + \iO w^2 \varphi dx &= \iO uw \varphi dx - \iO w \nabla w \cdot \nabla \varphi dx\\
&=\iO uw \varphi dx +\dfrac{1}{2} \iO  w^2  \Delta \varphi dx.
\end{align*}
By virtue of Lemma \ref{lem:30} and \eqref{eq:h3}, we apply the H\"older inequality and the Young inequality for the right-hand side in the above equality to obtain
\begin{align}
&(4A^2 + 1)\iO |\nabla w|^2 \varphi dx + (4A^2 + 1)\iO w^2 \varphi dx\notag\\
 &\leq (4A^2 + 1)\left(\iO e^{-v}u^2\varphi dx\right)^\frac{1}{2}\left(\iO w^2 e^{v}\varphi dx\right)^\frac{1}{2} + \dfrac{4A^2 + 1}{2}\|\Delta \varphi\|_{L^\infty(\Omega)} \iO w^2 dx\notag\\
&\leq \dfrac{1}{8}\iO e^{-v}u^2 \varphi dx +2(4A^2 + 1)^2  \iO w^2e^v \varphi dx + C\notag\\
&\leq \dfrac{1}{8}\iO e^{-v}u^2 \varphi dx + 2(4A^2 + 1)^2\left(\iO e^{av}\varphi dx\right)^\frac{1}{a}\left(\iO w^{\frac{2a}{a-1}}\varphi dx \right)^\frac{a-1}{a} + C\notag.
\end{align}
Noticing the assumption \eqref{ite:1}, we gain
\begin{align*}
(4A^2 + 1)\iO |\nabla w|^2 \varphi dx + (4A^2 + 1)\iO w^2 \varphi dx &\leq \dfrac{1}{8}\iO e^{-v}u^2 \varphi dx + C.
\end{align*}
Summarizing the above estimate into \eqref{ite:17}, we confirm from Lemma \ref{lem:30} that 
\begin{align*}
&\dfrac{1}{2}\dfrac{d}{dt}\iO |\nabla w|^2 \varphi dx + \dfrac{1}{2}\dfrac{d}{dt}\iO w^2 \varphi dx + \dfrac{1}{2}\iO e^{-v}u^2 \varphi dx\\
&\hspace{0.5cm}+ 
(4A^2 + 1)\iO |\nabla w|^2 \varphi dx + (4A^2 + 1)\iO w^2 \varphi dx
\notag\\
&\leq \dfrac{1}{4}\int_{B_{2r}} e^{-v}u^2 dx + 4A^2\iO |\nabla w|^2 \varphi dx + \dfrac{1}{8}\iO e^{-v}u^2 \varphi dx + C\\
&\leq \dfrac{3}{8}\int_{B_{2r}} e^{-v}u^2 dx + 4A^2\iO |\nabla w|^2 \varphi dx + C\\
&\leq \dfrac{3}{8}\int_{B_{r}} e^{-v}u^2 dx + \dfrac{3}{8}\int_{B_{2r} \setminus B_r}e^{-v}u^2 dx + 4A^2\iO |\nabla w|^2 \varphi dx + C\\
&\leq \dfrac{3}{8}\int_{B_{r}} e^{-v}u^2 dx + \dfrac{3}{8}\int_{B_{2r} \setminus B_r}u^2 dx + 4A^2\iO |\nabla w|^2 \varphi dx + C
\end{align*}
and refining the above inequality, we can rewrite as follows:
\begin{align*}
&\dfrac{1}{2}\dfrac{d}{dt}\iO |\nabla w|^2 \varphi dx + \dfrac{1}{2}\dfrac{d}{dt}\iO w^2 \varphi dx + \dfrac{1}{8}\int_{B_r} e^{-v}u^2 dx + 
\dfrac{1}{2}\iO |\nabla w|^2 \varphi dx + \dfrac{1}{2}\iO w^2 \varphi dx
\notag\\
&\leq \dfrac{3}{8}\int_{B_{2r} \setminus B_r}u^2 dx + C.
\end{align*}
By introducing 
\begin{align*}
Y(t) := \iO |\nabla w|^2 \varphi dx + \iO w^2 \varphi dx,
\end{align*}
and more multiplying $e^t$ and integrating with respect to time, we arrive at
\begin{align*}
Y(t) + \dfrac{1}{4}\int_0^t e^{s-t}\int_{B_r}e^{-v}u^2 dx ds \leq \dfrac{3}{4}\int_0^t e^{s-t}\int_{B_{2r}\setminus B_r} u^2 dx ds + C.
\end{align*}
Hence Lemma \ref{lem:bdd1} contributes to achieving that
\begin{align*}
Y(t) + \dfrac{1}{4}\int_0^t e^{s-t}\int_{B_r}e^{-v}u^2 dx ds \leq C,
\end{align*}
which is the desired our claim \eqref{ite:2}. Next we will show that for any $\alpha \in [1,\infty)$,
\begin{align}
\sup_{t > 0}\iO e^{\alpha v} \varphi dx < \infty.\label{ite:18}
\end{align} 
According to \eqref{bdd:v}: for all $(x,t) \in \Omega \times [0,\infty)$,
\begin{align*}
v(x,t) \leq \dfrac{1}{1-e^{-v_*}}\left(w(x,t) + K\right),
\end{align*}
this inequality allows us to proceed as follows:
\begin{align*}
\iO e^{\alpha v} \varphi &\leq \iO e^{\frac{\alpha}{1-e^{-v_*}}(w + K)} \varphi dx\\
&= e^{\frac{\alpha}{1-e^{-v_*}}K}\iO e^{\frac{\alpha}{1-e^{-v_*}}w} \varphi dx.
\end{align*}
We also make much use of \cite[Lemma 5.3]{NSS2000}, which is the modified Trudinger--Moser inequality (see \cite{CY1988, MJ1970}),  together with the above inequality and \eqref{ite:2} to infer that for all $t \in (0,\infty)$,
\begin{align*}
\iO e^{\alpha v}\varphi dx \leq e^{\frac{\alpha}{1-e^{-v_*}}K} C_\mathrm{TM}\exp \left(\dfrac{\alpha^2}{16\pi(1-e^{-v_*})^2} \iO |\nabla w|^2 \varphi dx\right) < \infty
\end{align*}
with $C_\mathrm{TM} > 0$ depending only on $\Omega$. This clearly implies our claim. What is left is to show time-independent $L^\infty$-boundedness of solution $(u,v)$. We shall have established the lemma if we prove the following: 
\begin{align*}
\sup_{t > 0}\|w(t)\|_{L^\infty(B_r)} < \infty.
\end{align*}
Indeed, the above result leads to 
\begin{align*}
\sup_{ t >0}\|v(t)\|_{L^\infty(B_r)} < \infty
\end{align*}
thanks to \eqref{bdd:v}, thereby we can show the conclusion using the Alikakos--Moser iteration scheme (see \cite{A1979}). Here, multiplying $-\Delta w + w= u$ by $\varphi$ redefining $\varphi$ as $\varphi = \varphi_{(0,\frac{r}{2},n)}$, we have
\begin{align*}
u\varphi &= w\varphi - (\Delta w)\varphi\notag\\
&= w\varphi - \Delta (w\varphi) + 2\nabla w \cdot \nabla \varphi + w \Delta \varphi.
\end{align*}
As a consequence of the elliptic regularity theorem and the Sobolev inequality, it follows that
\begin{align*}
\|w\varphi\|_{L^\infty(\Omega)} &\leq C_\mathrm{Sob}\|w\varphi\|_{W^{2,\frac{3}{2}}(\Omega)}\notag\\
&\leq C_\mathrm{Sob}C_\mathrm{Ell}\|u\varphi - 2\nabla w\cdot \nabla \varphi - w\Delta \varphi\|_{L^\frac{3}{2}(\Omega)}\notag\\
&\leq C_\mathrm{Sob}C_\mathrm{Ell}\|u\varphi\|_{L^\frac{3}{2}(\Omega)} + 2C_\mathrm{Sob}C_\mathrm{Ell}\|\nabla w \cdot \nabla \varphi \|_{L^\frac{3}{2}(\Omega)} + C_\mathrm{Sob}C_\mathrm{Ell}\|w\Delta\varphi\|_{L^\frac{3}{2}(\Omega)}.
\end{align*}
Since Lemma \ref{eq:h1}, the mass conversation law \eqref{mass:1}, and Lemma \ref{lem:30} together result in
\begin{align*}
\|w\varphi\|_{L^\infty(\Omega)} &\leq  C_\mathrm{Sob}C_\mathrm{Ell}\|u\varphi\|_{L^\frac{3}{2}(\Omega)}\\
&\hspace{0.5cm}+ 2C_\mathrm{Sob}C_\mathrm{Ell}\|\nabla \varphi\|_{L^\infty(\Omega)} C\|u_0\|_{L^1(\Omega)} + C_\mathrm{Sob}C_\mathrm{Ell}\|\Delta \varphi\|_{L^\infty(\Omega)}C\|u_0\|_{L^1(\Omega)},
\end{align*}
we focus on the first term on the right-hand side of the above inequality. By using the H\"older inequality and the Young inequality, we deduce from \eqref{ite:18} that
\begin{align*}
\left(\iO (u\varphi)^\frac{3}{2} dx\right)^\frac{2}{3} &\leq \left(\int_{B_{r}} e^{-v}u^2 dx\right)^\frac{1}{2}\left(\iO e^{3v}\varphi^6 dx\right)^\frac{1}{6}\notag\\
&\leq C\left(\int_{B_{r}} e^{-v}u^2 dx\right)^\frac{1}{2}\\
&\leq \int_{B_{r}} e^{-v}u^2 dx + \dfrac{C}{4}.
\end{align*}
Therefore due to \eqref{ite:2}, we get for all $t \in (0,\infty)$,
\begin{align}
\int_0^t e^{s-t}\|w\varphi\|_{L^\infty(\Omega)} ds &\leq C_\mathrm{Sob}C_\mathrm{Ell}\int_0^t e^{s-t} \int_{B_{r}} e^{-v}u^2 dx + C\int_0^t e^{s-t} ds\notag\\
&\leq C.\label{ite:20}
\end{align}
Since \cite[Lemma 5]{FJ2021} makes it obvious that $w_t \leq w$ holds for all $(x,t) \in \Omega \times (0,\infty)$, we can calculate
\begin{align*}
\int_0^t \left(\dfrac{d}{ds}e^s w\right) ds \leq 2\int_0^t e^s w(x,s) ds.
\end{align*}
This implies that for $(x,t) \in \Omega \times (0,\infty)$,
\begin{align*}
w(x,t) &= e^{-t}w(x,0) + e^{-t} \int_0^t \left(\dfrac{d}{ds} e^s w(x,s)\right) ds\\
&\leq e^{-t}w(x,0) + 2\int_0^t e^{s-t}w(x,s) ds\notag\\
&= e^{-t}(I -\Delta)^{-1}u_0(x)+ 2\int_0^t e^{s-t}w(x,s) ds.
\end{align*}
Consequently we can conclude from \eqref{ite:20} and Lemma \ref{lem:30} that
\begin{align*}
\sup_{t > 0}\|w(t)\|_{L^\infty(B_\frac{r}{2})} < \infty,
\end{align*}
which implies the end of the proof. In other words, if the origin is the grow-up point, \eqref{ite:0} holds.

\end{proof}
\begin{proof}[Proof of Theorem \ref{th:1}]
Of course, since it follows that
\begin{align*}
\iO e^{av}\varphi dx \leq \iO e^{av} dx,
\end{align*}
Proposition \ref{prop:3.11} explicitly formulates the theorem.
\end{proof}

\subsection{Proof of Corollary \ref{th:2} and Corollary \ref{th:3}}\label{subsec:3}
In this part, we will look more closely at the behavior of the radially symmetric solution to \eqref{p} around the origin, with the aim of demonstrating the concentration phenomena occurring there. 
As a first step, we will present that the solution satisfies the H\"older condition in $x$ and also $t$. 

\begin{lemma}\label{lem:H}
Let $(u,v)$ be a radially symmetric grow-up solution of \eqref{p}. Then for any $0 < r \ll 1$ and $\alpha \in (0,1)$, there exists a positive constant $C = C(r, \alpha)$ such that
\begin{align*}
\|(u,v)\|_{C^{2+\alpha, 1 + \frac{\alpha}{2}}((\overline{\Omega} \setminus B_r) \times (0,\infty))} \leq C(r,\alpha).
\end{align*}
\end{lemma}

\begin{proof}
We can derive from the definition of grow-up point and Proposition \ref{pro:rad0} that for any $0 <r \ll 1$, there exists a positive constant $C(r)$ depending on $r$ such that
\begin{align*}
\sup_{t > 0}\|u(t)\|_{L^\infty(\Omega \setminus B_r)} \leq C(r).
\end{align*}
In addition, we apply \cite[(STEP 3) in Proposition 3.7]{YS2025} for the second equation of \eqref{p} to get
\begin{align}
\sup_{t > 0}\|v(t)\|_{C^1(\overline{\Omega} \setminus B_{r^\prime})} \leq C(r)\label{Hellder_1}
\end{align}
for any $0 < r < r^\prime \ll 1$. By noticing that Lemma \ref{lem:below} ensures the lower bound of $v$ so that 
\begin{align*}
v(x,t) \ge v_*\quad \mathrm{for\ all}\ (x,t) \in (\overline{\Omega}\setminus B_{r^\prime}) \times (0,\infty),
\end{align*}
\cite[Theorem 10.1]{LSU1968} and \eqref{Hellder_1} enable us to get
\begin{align*}
\|u\|_{C^{\alpha, \frac{\alpha}{2}}((\overline{\Omega} \setminus B_{r^{\prime\prime}}) \times (0,\infty))} \leq C(r)
\end{align*}
for any $0 < r < r^\prime < r^{\prime\prime} \ll 1$. Moreover we again apply \cite[Theorem 10.1]{LSU1968} for the second equation of \eqref{p} to obtain
\begin{align*}
\|v\|_{C^{\alpha, \frac{\alpha}{2}}((\overline{\Omega} \setminus B_{r^{\prime\prime\prime}}) \times (0,\infty))} \leq C(r)
\end{align*}
for any $0 < r < r^\prime < r^{\prime\prime} < r^{\prime\prime\prime}\ll 1$. Thanks to the lower bound of $v$,
we can complete the proof by repeating this argument once again.
\end{proof}

We next introduce a Lyapunov functional focusing on a localized area, which is so called `` localized Lyapunov functional.'' The combination of a boundedness of the localized Lyapunov functional and a modified Trudinger--Moser inequality (see \cite[Lemma 5.3]{NSS2000}) concludes the proof of Corollary \ref{th:2} following the similar approach as \cite[Subsection 3.5]{YS2025}, thereby we regard Corollary \ref{th:2} as proved through the establishment of the following proposition, without providing the detailed proof.

\begin{proposition}
Let $0 < r \ll 1$ and $n$ be a sufficiently large, and more let $\varphi = \varphi_{(0,r,n)}$ be as in Lemma \ref{lem:30}. Then there exists a positive constant $C$ such that for all $t \in (0,\infty)$,
\begin{align*}
\dfrac{d}{dt}\mathcal{F}_\varphi (t) + \mathcal{F}_\varphi (t) + \mathcal{D}_\varphi (t) \leq C,
\end{align*}
where
\begin{align*}
\mathcal{F}_\varphi (t) &:= \iO (u\log u - uv)\varphi dx + \dfrac{1}{2}\iO |\nabla v|^2 \varphi dx + \iO v^2 \varphi dx,\\
\mathcal{D}_\varphi (t) &:= \iO ue^{-v}|\nabla (\log u - v)|^2 \varphi dx + \iO |v_t|^2 \varphi dx.
\end{align*}
Moreover, it holds that
\begin{align*}
\sup_{t > 0}\mathcal{F}_\varphi (t) < \infty.
\end{align*}
\end{proposition}

\begin{proof}
Multiplying the first equation of \eqref{p} by $(\log u - v)\varphi$ and integrating over $\Omega$, we have
\begin{align*}
\iO u_t (\log u - v)\varphi dx = \iO \nabla \cdot (e^{-v} (\nabla u - u\nabla v))(\log u - v)\varphi dx
\end{align*}
and also we integrate by parts to obtain
\begin{align*}
&\iO u_t (\log u - v)\varphi dx\\
&= -\iO ue^{-v} |\nabla (\log u - v)|^2 \varphi dx - \iO e^{-v}(\log u - v)(\nabla u - u\nabla v) \cdot \nabla \varphi dx.
\end{align*}
Since the left-hand side on the above equality can be rewritten as
\begin{align*}
\iO u_t (\log u - v)\varphi dx = \dfrac{d}{dt}\iO (u\log u - uv)\varphi dx - \iO u_t\varphi dx + \iO uv_t\varphi dx,
\end{align*}
from the second equation of \eqref{p} and by integrating by parts, we calculate
\begin{align*}
\iO uv_t\varphi dx &= \iO (v_t - \Delta v + v) v_t \varphi dx\\
&= \iO |v_t|^2 \varphi dx + \dfrac{1}{2}\dfrac{d}{dt}\iO |\nabla v|^2 \varphi dx + \iO v_t \nabla v \cdot \nabla \varphi dx + \dfrac{1}{2}\dfrac{d}{dt}\iO v^2 \varphi dx.
\end{align*}
By using the first equation of \eqref{p} and integrating by parts, the term $\iO u_t \varphi dx$ can be revised as
\begin{align*}
\iO u_t\varphi dx &= \iO \nabla \cdot (e^{-v}(\nabla u - u\nabla v))\varphi dx\\
&= \iO e^{-v}(\nabla u - u\nabla v) \cdot \nabla \varphi dx.
\end{align*}
Gathering this information implies that
\begin{align*}
\dfrac{d}{dt}\mathcal{F}_\varphi (t) + \mathcal{D}_\varphi (t) = \mathcal{R}_\varphi (t),
\end{align*}
where 
\begin{align*}
\mathcal{R}_\varphi (t) := \iO (e^{-v}(\nabla u - u\nabla v) - e^{-v}(\log u - v)(\nabla u - u\nabla v) - v_t\nabla v)\cdot \nabla \varphi dx.
\end{align*}
Thanks to Lemma \ref{lem:30} and Lemma \ref{lem:H}, the remainder term $\mathcal{R}_\varphi$ is uniformly bounded, that is to say, it holds that for $t \in (0,\infty)$
\begin{align}
\dfrac{d}{dt}\mathcal{F}_\varphi (t) + \mathcal{D}_\varphi (t) \leq C,\label{lyalocal}
\end{align}
with $C := \sup_{t > 0}\mathcal{R}_\varphi (t) < \infty$. What is left is to show a time-independent estimate for $\mathcal{F}_\varphi (t)$. Since $\F$ is the Lyapunov functional satisfying the energy identity \eqref{lya}, it holds that
\begin{align*}
\F (t) \leq \F (0).
\end{align*}
Therefore Lemma \ref{lem:30} and Lemma \ref{lem:H} lead to
\begin{align*}
\mathcal{F}_\varphi (t) &\leq \mathcal{F}(t) + C\\
&\leq \mathcal{F}(0) + C,
\end{align*}
where $C$ is a positive constant, which is shown by \cite[Lemma 3.10]{YS2025}. This completes the proof by adding $\F_\varphi (t)$ to \eqref{lyalocal}.
\end{proof}

\begin{proof}[Proof of Corollary \ref{th:3}]
We notice that the stationary problem for our system \eqref{p} coincides with that of the Keller--Segel system. Therefore we can give full play to the results in \cite[Subsection 3.6]{YS2025} thanks to Lemma \ref{lem:H} and the assumption \eqref{eq:i3} encouraging a convergence to a stationary solution of \eqref{p} in the domain excluding the origin. Hence we skip the detailed proof for simplicity.
\end{proof}

\textbf{Acknowledgments}

The author wishes to thank my supervisor Kentaro Fujie in Tohoku University for several useful comments concerning this paper.

\bigskip
\address{ 
Mathematical Institute \\
Tohoku University \\
Sendai 980-8578 \\
Japan
}
{soga.yuri.q6@dc.tohoku.ac.jp}

\end{document}